\documentclass[final]{siamltex}

\usepackage{amsmath}
\usepackage{amssymb}
\usepackage{graphicx}
\usepackage{color} 
\graphicspath{{FractureFig/}}

\newcommand{\bI}{\pmb{I}}
\newcommand{\bK}{\pmb{K}}

\newcommand{\bP}{\pmb{P}}

\newcommand{\bu}{\pmb{u}}
\newcommand{\bv}{\pmb{v}}
\newcommand{\bx}{\pmb{x}}
\newcommand{\bn}{\pmb{n}}
\newcommand{\Div}{\text{div}\;}
\newcommand{\Divtau}{\text{div}_{\tau} \;}
\newcommand{\HicF}{{H_*^{1}}}
\newcommand{\MT}{{M_{\cal O}}}
\newcommand{\ST}{{\Sigma_{\cal O}}}

\newcommand{\iO}{\mathcal{O}}

\newcommand{\iS}{\mathcal{S}}
\newcommand{\iT}{\mathcal{T}}

\newcommand{\mR}{\mathbb{R}}
\newcommand{\R}{\mathbb{R}}
\newcommand{\ds}{\displaystyle}
\newcommand{\Hoo}{{H_{*}^{1,1}}}

\newtheorem{remark}[theorem]{Remark}

\definecolor{darkcyan}{RGB}{0,139,139}
\definecolor{mediumred-violet}{rgb}{0.73, 0.2, 0.52}
\definecolor{americanrose}{rgb}{1.0,0.01,0.4}
\definecolor{shamrockgreen}{rgb}{0.0, 0.62, 0.38}

\title{Space-time Domain Decomposition and Mixed Formulation for reduced fracture models}

\author{Thi-Thao-Phuong Hoang \footnotemark[1]\ \footnotemark[2]\ \footnotemark[4]   
\and  Caroline Japhet\footnotemark[3]\ \footnotemark[1]
\and  Michel~Kern\footnotemark[1]\ 
\and  Jean~E.~Roberts\footnotemark[1]}

\begin{document}
\maketitle

\renewcommand{\thefootnote}{\fnsymbol{footnote}}

\footnotetext[1]{INRIA Paris-Rocquencourt, 78153 Le Chesnay Cedex, France~(\texttt{Phuong.Hoang\_Thi\_Thao@inria.fr},
  \texttt{Michel.Kern@inria.fr}, \texttt{Jean.Roberts@inria.fr})}
\footnotetext[3]{Universit\'e Paris 13, UMR 7539, LAGA,
  99 Avenue J-B Cl\'ement, 93430 Villetaneuse, France (\texttt{japhet@math.univ-paris13.fr}).}
\footnotetext[2]{Current address: Ho Chi Minh City University of Pedagogy, Vietnam.}
\footnotetext[4]{Partially supported by ANDRA, the French agency for nuclear waste management}
\footnotetext[1]{Partially supported by GNR MoMaS.}

\renewcommand{\thefootnote}{\arabic{footnote}}

\begin{abstract}
In this paper we are interested in the "fast path" fracture and we aim to use global-in-time, nonoverlapping domain decomposition methods to model flow and transport problems in a porous medium containing such a fracture. We consider a reduced model in which the fracture is treated as an interface between the two subdomains. 
 Two domain decomposition methods  are considered: one
uses the time-dependent Steklov–Poincar\'e operator and the other uses optimized Schwarz waveform
relaxation (OSWR) based on Ventcell transmission conditions. For each method, a mixed formulation
of an interface problem on the space-time interface is derived, and different time
grids are employed to adapt to different time scales in the subdomains and in the fracture.
Demonstrations of the well-posedness of the Ventcell subdomain problems is given for the mixed formulation.
An analysis for the convergence factor of the OSWR algorithm is given in the case with fractures 
to compute the optimized parameters. Numerical results for two-dimensional
problems with strong heterogeneities are presented to illustrate the performance of the two methods.
\end{abstract}



\pagestyle{myheadings}
\thispagestyle{plain}
\markboth{T.T.P Hoang, C. Japhet, M. Kern and J.E. Roberts}{Space-time DD for reduced fracture models}

\section{Introduction}
In many simulations of time-dependent physical phenomena, the domain of calculation is a union of domains with different physical properties and in which the lengths of the domains and the time scales may be very different. In particular, this is the case for a domain where there exist fractures and faults. In such a case, the fluid flows rapidly through these paths while it moves much more slowly through the rock matrix. As a result, the contaminants present in the porous medium that travel with the fluid are transported faster than in the case when there is no fracture. Thus the time scales in the fractures and in the surrounding medium are very different, and in the context of simulation, one might want to use much smaller time steps in the fractures than in the rock matrix. For simplicity we consider the case in which the domain is separated into two matrix subdomains by a fracture. The permeability in the fracture can be larger or smaller than that in the surrounding medium. A large permeability fracture corresponds to a fast pathway and
a small permeability fracture corresponds to a geological barrier. Here we are interested in the ``fast path'' fracture.
Modeling flow in porous media with fractures is challenging and requires a multi-scale approach: first, the fractures represent strong heterogeneities as they have much higher or much lower permeability than that in the surrounding medium; second, the fracture width is much smaller than any reasonable parameter of spatial discretization. Thus, to tackle the problem, one might need to refine the mesh locally around the fractures. However, this is well-known to be very computationally costly and is not useful at the macroscopic scale (i.e. when the fractures can be modeled individually). One possible approach is to treat the fractures as domains of co-dimension one, i.e. interfaces between subdomains
(see~\cite{Alboin,Angot09,Helmig99,Faille02,Alessio12,Vincent,Morales10,Morales12,Tunc12} and the references therein) so that one can avoid refining locally around the fractures. We point out that in these reduced fracture models, unlike in some discrete fracture models, interaction between the fractures and the surrounding porous medium is taken into account.

We are concerned with algorithms for modeling flow and transport in porous media containing such fractures.
In particular, in this article we investigate two space-time domain decomposition methods, well-suited to nonmatching time grids. We use mixed finite elements~\cite{brezzi1991mixed,RobertsThomas} as they are mass conservative and they handle well heterogeneous and anisotropic diffusion tensors.

The first method is a global-in-time preconditionned Schur method (GTP-Schur) which uses a Steklov–Poincar\'e-type
operator. For stationary problems, this kind of method
(see~\cite{Mathew:DDM:2008,quarteroni2008numerical,Toselli:DDM:2005})
is known to be efficient for problems with strong heterogeneity. It uses the so-called
balancing domain decomposition (BDD) preconditioner introduced and analyzed in~\cite{Mandel,Mandelweights}, and in~\cite{CowsarBDD} for mixed finite elements. It involves at each
iteration the solution of local problems with Dirichlet and Neumann data
and a coarse grid problem to propagate information globally and to
ensure the consistency of the Neumann subdomain problems.
An extension to the case of unsteady problems with the
construction of the time-dependent Steklov-Poincar\'e operator was introduced in~\cite{PhuongThesis,PhuongSINUM},
where  an interface problem on the space-time interfaces between subdomains is derived.
However, for the time-dependent Neumann-Neumann problems there are no difficulties concerning consistency, 
and we are dealing with only a small number of subdomains, so we consider only a Neumann-Neumann type
preconditioner, an extension to the nonsteady case of the method of~\cite{NNPrecond}.
A Richardson iteration for the primal formulation was independently introduced in~\cite{gander2014dirichlet,Kwok},
  and its convergence was analyzed.
 In the case of elliptic problems with fractures, a local preconditionner
\cite{Laila} significantly improves the convergence of the method.

The second method  is a global-in-time optimized Schwarz method (GTO-Schwarz) and uses the optimized Schwarz
waveform relaxation (OSWR) approach.
The OSWR and GTP-Schur methods are iterative methods that compute in the subdomains
over the whole time interval, exchanging space-time boundary data through transmission conditions
  on the space-time interfaces.
  The OSWR algorithm uses
more general (Robin or Ventcell) transmission operators in which coefficients can be
optimized to improve convergence rates, see~\cite{OSWRwave,JaphetDD9,VMartin}.
The optimization of the Robin (or Ventcell)
parameters was analyzed
in~\cite{Bennequin} and the optimization method was extended to the case of
discontinuous coefficients
in~\cite{PMThesis,BertheDD21,OSWR2d,BlayoHJ,OSWR1d2,PhuongThesis,PhuongSINUM}. Generalizations to heterogeneous problems with nonmatching time grids were introduced
in~\cite{PMThesis,BertheDD21,BlayoHJ,OSWR1d2,Haeberlein,OSWR3sub,OSWRDG,OSWRDG2,PhuongThesis,PhuongSINUM}. More
precisely, in~\cite{BlayoHJ,OSWRDG,OSWRDG2}, a discontinuous Galerkin (DG) method for the time discretization of the
OSWR algorithm was introduced and analyzed for the case of nonconforming time grids.
A suitable time projection between subdomains is defined using an optimal projection algorithm as
in~\cite{Projection2d:3d,Projection1d} with no additional grid.
The classical Schwarz algorithm for stationary problems with mixed finite elements was analyzed
in~\cite{JeanRobinmixed}.
An OSWR method with Robin transmission
conditions for a mixed formulation was proposed and analyzed in~\cite{PhuongThesis,PhuongSINUM}, where
a mixed form of an interface problem on the space-time interfaces between subdomains was derived.
In~\cite{HJKRDD22}, an Optimized Schwarz method with Ventcell conditions
in the context of mixed formulations was proposed.
This method is not obtained in such a straightforward manner as in the case of primal formulations
as Lagrange multipliers have to be introduced on the interfaces to handle tangential derivatives involved in
the Ventcell conditions. 

In this work, we  define both a GTP-Schur and a GTO-Schwarz algorithm for a problem modeling flow of a single phase,
compressible fluid in  a porous medium with a fracture.
A straightforward application of~\cite{PhuongSINUM} would be to consider the fracture as a third subdomain and to take smaller time steps there.  We consider instead however a reduced model in which the fracture is treated as an interface between two subdomains.

The definition of the GTP-Schur method is a straightforward extension of that in~\cite{PhuongSINUM}.
However, to define the GTO-Schwarz method, something more is needed:
a linear combination between the pressure continuity
equation and the fracture problem is used as a transmission condition (which leads naturally to Ventcell conditions),
and a free parameter is used to accelerate the convergence rate.
The well-posedness of the subdomain
problems involved in the first approach was addressed in~\cite{BoffiGastaldi04,PhuongSINUM,Arbogast},
using Galerkin's method and suitable a priori estimates. In this paper,
the proof of well-posedness of both the coupled model and the Ventcell subdomain problems involved in the GTO-Schwarz approach is shown to follow from a more general theorem that covers the two cases.

Note that more general reduced models that can handle both large and small permeability fractures \cite{Vincent} introduce more complicated transmission conditions on the fracture-interface (in the form of Robin type conditions, where
the Robin coefficient has a physical origin), and it is not yet clear how to formulate an associated domain decomposition problem with a parameter that can be optimized. 

\bigskip

This paper is organized as follows: in the remainder of the introduction
(Subsection~\ref{sec:abstractresult}), we state an abstract existence and uniqueness theorem
for evolution problems in mixed form, the proof being deferred to Appendix~\ref{sec:append}.
In Section~\ref{A3Sec:Reduced}, we consider a reduced model with a highly permeable fracture and
prove its well-posedness. Then in Section~\ref{A3Sub:SchurM1} we consider the GTP-Schur approach,
based on physical transmission
conditions, for solving the resulting problem. Different preconditionners for this method are proposed.
In Section~\ref{A3subsec:VentcellDiff} we consider the GTO-Schwarz method,
based on more general (e.g. Ventcell) transmission
conditions, for solving the resulting problem. 
We prove the well-posedness of the subdomain problems
with Ventcell boundary conditions.
In Section~\ref{A4Subsec:TimeNonc} we consider the semi-discrete problems in time using different
time grids in the subdomains. Finally, in Section~\ref{A3Sec:Num}, results of two-dimensional (2D) numerical
experiments comparing the different methods are discussed.

\subsection{Abstract evolution problems in mixed form}
\label{sec:abstractresult}
The goal of this section is to give an existence and uniqueness result
for evolution problems posed in mixed form, in the spirit of the
well-known theorem for weak parabolic problems (see for
example~\cite[vol. 5]{lions2000mathematical}).

We consider two Hilbert spaces $\Sigma$, and $M$ ($M$ will be identified
with its dual), and assume we have continuous bilinear forms
\begin{equation*}
    a :  \Sigma \times \Sigma   \longrightarrow  \R, \qquad
    b :   \Sigma \times M   \longrightarrow  \R, \qquad
    c :  M \times M  \longrightarrow \R 
\end{equation*}
and a continuous linear form
\begin{equation*}
  L: M \longrightarrow \R.
\end{equation*}

We study here an abstract version of a parabolic problem in mixed
form: 
\begin{eqnarray}
\mbox{Find  $ p \in H^1(0,T;M) $ and $ \bu  \in L^2(0,T;\Sigma) $ such that,}
   \nonumber\\
\begin{array}{rlll}
a(\bu, \bv) - b(\bv, p) &=&0, & \forall \bv \in \Sigma,  \hspace{2cm}\\
(\partial_{t} p, \mu)_M + c(p, \mu) + b (\bu, \mu) & = & L (\mu), & \forall \mu \in M,\\
p(\cdot, 0) &= & p_0,&
\end{array}
\label{App31dWeak}
\end{eqnarray}
for some $p_{0} \in M$.

\medskip
We make the following hypotheses on the data:
\begin{itemize}
\item
The bilinear form~$a$ is positive definite on $\Sigma$:
  \begin{equation}
    \tag{H1} \label{eq:H2}
a(\bu, \bu) > 0 \qquad \forall \bu \in \Sigma, \ \ \bu \neq 0,
  \end{equation}
so that $a$ defines a norm on $\Sigma$, and we denote by $ \Sigma_a$ the space $ \Sigma$ with the norm induced by the bilinear form $a$.
Note however that this norm will
not necessarily be equivalent to the initial norm on $\Sigma$.
\item 
The bilinear form~$c$ is positive semidefinite on $M$
  \begin{equation}
    \tag{H2} \label{eq:H3}
    c(p, p) \ge 0 \qquad \forall p \in M.
  \end{equation}
\item 
The bilinear forms~$a$ and~$b$ satisfy the following compatibility
condition: there exists $\beta >0$ such that
\begin{equation}
  \tag{H3} \label{eq:H4}
  \forall \bu \in \Sigma, \quad \sup_{\mu \in M} \dfrac{b(\bu,
    \mu)^2}{\| \mu\|_M^2} + \| \bu \|_{\Sigma_a}^2 \geq \beta \| \bu \|_\Sigma^2.
\end{equation}
\item 
There exists a subspace $W \subset M$ (with continuous embedding) on which 
the bilinear form~$b$ satisfies the stronger
continuity property : there exists $C_b>0$ such that
\begin{equation}
  \tag{H4} \label{eq:H5}
 b(\bu, \mu) \leq C_b \| \bu \|_{\Sigma_a} \| \mu \|_W, \quad \forall \bu \in
 \Sigma \text{ and } \forall \mu \in W.
\end{equation}
\end{itemize}

In most cases, the application of hypothesis~\eqref{eq:H4} will appear
in a more natural form if it is written using the operator $B: W \to
M$ associated with the bilinear form~$b$, that is such that
\begin{equation*}
  \forall \bu \in W, \, \mu \in M, \quad b(\bu, \mu) = (B\bu, \mu)_M.
\end{equation*}
Then, hypothesis~\eqref{eq:H4} can be written in the equivalent form: there exists $\beta >0$ such that
  \begin{equation}
    \tag{H3'} \label{eq:H4p}
 \forall \bu \in \Sigma, \quad \|B \bu \|^2_M +
  \| \bu \|_{\Sigma_a}^2 \geq \beta \| \bu \|_\Sigma^2,
  \end{equation}
see the following remark.
\begin{remark}
\label{rem:hyph4}
  Hypothesis~\eqref{eq:H4} is not equivalent to the inf-sup
  condition. The inf-sup condition expresses the surjectivity of
  $B^T$, whereas here we need the ellipticity of $B$ with respect to the norm
defined by $a$. This also implies a form of compatibility between $a$
and $B$.
  This implies that a is elliptic on the kernel of $B$, i.e. if $B\bu = 0$,
    $a(\bu,\bu) \ge \beta \|\bu\|_{\Sigma}^2$.
\end{remark}

\medskip
The basic existence and uniqueness result for
problem~\eqref{App31dWeak} is the following:

\begin{theorem}
\label{thm:abst}
Let $M$ and $\Sigma$ be Hilbert spaces, and let a, b and c be continuous, bilinear forms satisfying
\eqref{eq:H2} through \eqref{eq:H4}. Then, if $L$ is a continuous linear form on $M$ and $p_0 \in W$,
where $W \subset M$ satisfies \eqref{eq:H5}, 
then problem~\eqref{App31dWeak}
has a unique solution, for which the following estimate holds: 
\begin{equation}
  \label{eq:absestim}
    \| \bu \|_{L^2(0,T; \Sigma)} + \|p\|_ {L^\infty(0, T; M)} +
      \| \partial_t p \|_{L^2(0, T; M)}  \leq C \left( \| L \|^2_{L^2(0,T; M)²} +
      \|p_0\|_W^2 \right).
\end{equation}
  
\end{theorem}

The proof of the theorem will be given in Appendix~\ref{sec:append}.

\begin{remark}
  The case $c=0$ \emph{is} allowed, and is actually the most common
  case (cf Theorem~\ref{A3thrm}).
\end{remark}

\begin{remark}
This result is a generalisation to the abstract setting of Lemma~(3.1)
in~\cite{Arbogast}. This problem has also been considered by Boffi and
Gastaldi~\cite{BoffiGastaldi04}, but the estimates given there
(without proof) are different: they
dispense with the regularity requirement $p_0 \in W$,
at the expense of introducing weighted estimates in time to cope with
the possibility of a singularity at the initial time.
A proof in a more general setting is given in~\cite{feecparabolic}. However this proof uses semigroup theory
  (see Theorem~4.1 of~\cite{feecparabolic}), while the one we propose in this paper is with a priori estimates, in
  the same spirit as in~\cite{Arbogast}.
\end{remark}

We give a simple application of Theorem~\ref{thm:abst} (other
 applications will be given in Theorems~\ref{A3thrm}
and~\ref{A3thrmVentcell} below).

We consider the heat equation with Dirichlet boundary conditions in mixed form. For a domain $\Omega
\subset \R^d$ ($d=2 \text{ or }3$) and $T>0$, we look for $p: \Omega
\times [0, T] \longrightarrow \R$, solution of:
\begin{equation}
\label{eq:heat}
  \begin{aligned}
    \dfrac{\partial p}{\partial t} - \Delta p &= f  & \text{ in }& \Omega
    \times [0, T] \\
    p &= 0 & \text{ on }& \partial \Omega \times [0, T] \\
    p(x, 0) &= p_0(x) & \text{ in }& \Omega.
  \end{aligned}
\end{equation}

To obtain the mixed form of~\eqref{eq:heat},  we define the
spaces $\Sigma = H(\Div, \Omega)$ and $M=L^2(\Omega)$, the bilinear
forms $a, \ b$ (here we will take $c=0$) and the linear form $L$
\begin{align}
  \label{eq:heatmix}
  a: \Sigma \times \Sigma \longrightarrow \R, & \quad a(\bu, \bv) = \int_\Omega \bu
  \cdot \bv \\
b: \Sigma \times M \longrightarrow \R, & \quad b(\bu, \mu)= \int_\Omega
\mu\,  \Div \bu,\\
  L:  M \rightarrow \R, &\quad L(\mu) = \int_\Omega f \mu.
\end{align}

To apply Theorem~\ref{thm:abst}, we check hypothesis~\eqref{eq:H2}
to~\eqref{eq:H5} above. This is trivial
for~\eqref{eq:H2} and~\eqref{eq:H3}. To
check~\eqref{eq:H4}, we use the equivalent
form~\eqref{eq:H4p}. Operator~$B$ is simply the divergence, so that 
\begin{equation*}
  \|B \bu \|^2_M +  \| \bu \|_{\Sigma_a}^2 = \int_\Omega \| \Div \bu
  \|^2 + \int_\Omega \| \bu \|^2 = \| \bu \|^2_\Sigma,
\end{equation*}
and~\eqref{eq:H4} is valid with $\beta=1$. 
Last we check that~\eqref{eq:H5} is also valid with
$W=H_0^1(\Omega)$. Using Green's formula, we obtain
\begin{equation*}
  b(\bu, \mu) = \int_\Omega \mu\, \Div \bu = - \int_\Omega \bu \cdot
  \nabla \mu, 
\end{equation*}
from which~\eqref{eq:H5} follows.

\section{A reduced fracture model} \label{A3Sec:Reduced}
For a bounded domain $ \Omega $ of $ \R^{d} \; (d=2,3) $ with Lipschitz boundary $ \partial \Omega $ and some fixed time $ T > 0 $, we consider the problem of flow of a single phase, compressible fluid
written in mixed form as follows:
\begin{equation} \label{A3primal}
\begin{array}{cll} s\partial_{t} p + \Div \bu &=q  & \text{in} \; \Omega \times (0,T), \\
 \bu  & = -\bK  \nabla p & \text{in} \; \Omega \times (0,T), \\
p & = 0 & \text{on} \; \partial \Omega \times (0,T),\\
p(\cdot, 0) & = p_{0} & \text{in} \; \Omega,
\end{array}\end{equation}
where $ p $ is the pressure, $ \bu $ the velocity, $ q $ the source term, $ s $ the storage coefficient and $ \bK $ a symmetric, time independent, hydraulic, conductivity tensor (see e.g. \cite{PhuongThesis}). 
For simplicity we have imposed a homogeneous Dirichlet condition on the boundary. 

We suppose that the fracture $ \Omega_{f} $ is a subdomain of $ \Omega $ , of thickness $ \delta $, that separates $ \Omega $ into two connected subdomains (see Figure~\ref{Fig:A3TwoModels}, left where for visualization purposes the size of $\delta $ is depicted as being relatively much larger than it is in reality),
$$ \Omega \setminus \overline{\Omega}_{f} = \Omega_{1} \cup \Omega_{2}, \quad \Omega_{1} \cap \Omega_{2} = \emptyset.
$$ 
Also, for simplicity, we assume that $ \Omega_{f} $ consists of the intersection with $ \Omega $ of a line or plane $ \gamma $ (depending on whether $ d=2 $ or $ 3 $), together with the points $ \bx = \bx_{\gamma} + s \bn $ where $ \bx_{\gamma} \in \gamma $, $ s\in \left (-\frac{\delta }{2}, \frac{\delta }{2}\right ) $ and $ \bn $ is a unit vector normal to $ \gamma $. 
We denote by $ \gamma_{i} $ the part of the boundary of $ \Omega_{i} $ shared with the boundary of the fracture~$ \Omega_{f} $:
$$ \gamma_{i} = \left (\partial \Omega_{i} \cap \partial \Omega_{f} \right) \cap \Omega, \quad i=1,2,
$$
and we denote by $ \bn_{i} $ the unit, outward pointing, normal vector field on $ \partial \Omega_{i} $.
\begin{figure}[htbp]
\vspace{-0.8cm}
\begin{minipage}[c]{0.4 \linewidth}
\setlength{\unitlength}{1pt} 
\begin{picture}(140,140)(0,0)
\thicklines
\put(0,0){\includegraphics[scale=0.55]{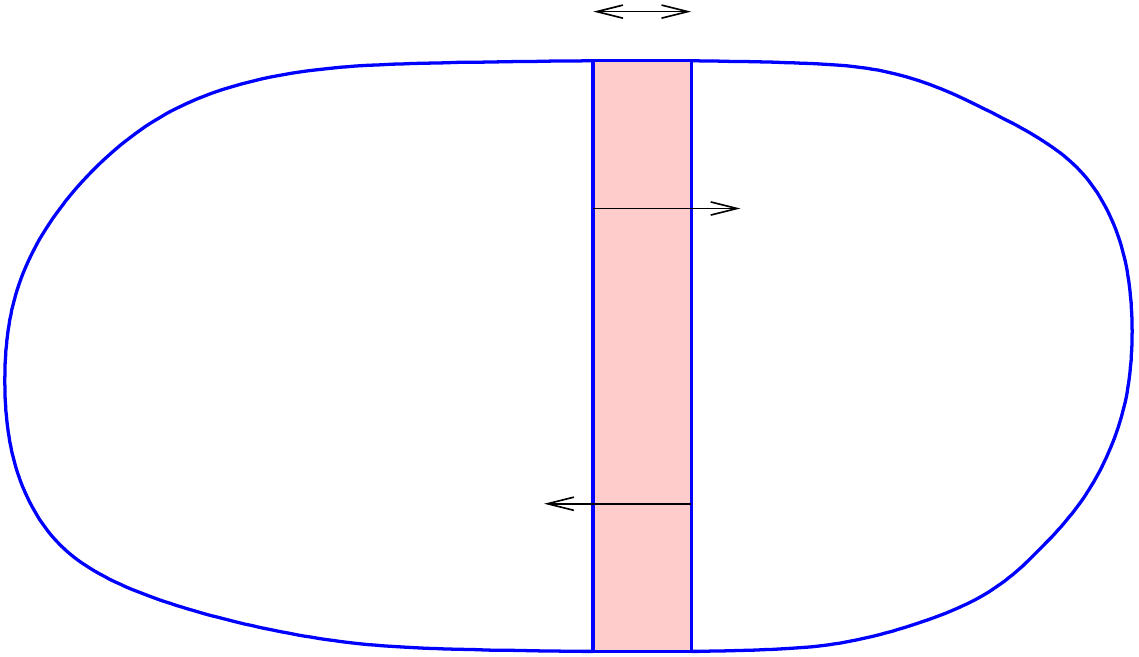} \\}
\put(40,40){$ \Omega_{1} $}
\put(135,40){$ \Omega_{2} $}
\put(120,77){$ \bn_{1}$}
\put(75,15){$ \bn_{2} $}
\put(80,56){$ \gamma_{1} $}
\put(115,56){$ \gamma_{2} $}
\put(97,40){$ \Omega_{f} $}
\put(100,107){$ \delta $}
\end{picture}
\end{minipage} \hspace{2cm}
\begin{minipage}[c]{0.4 \linewidth}
\setlength{\unitlength}{1pt} 
\begin{picture}(140,140)(0,0)
\thicklines
\put(0,0){\includegraphics[scale=0.55]{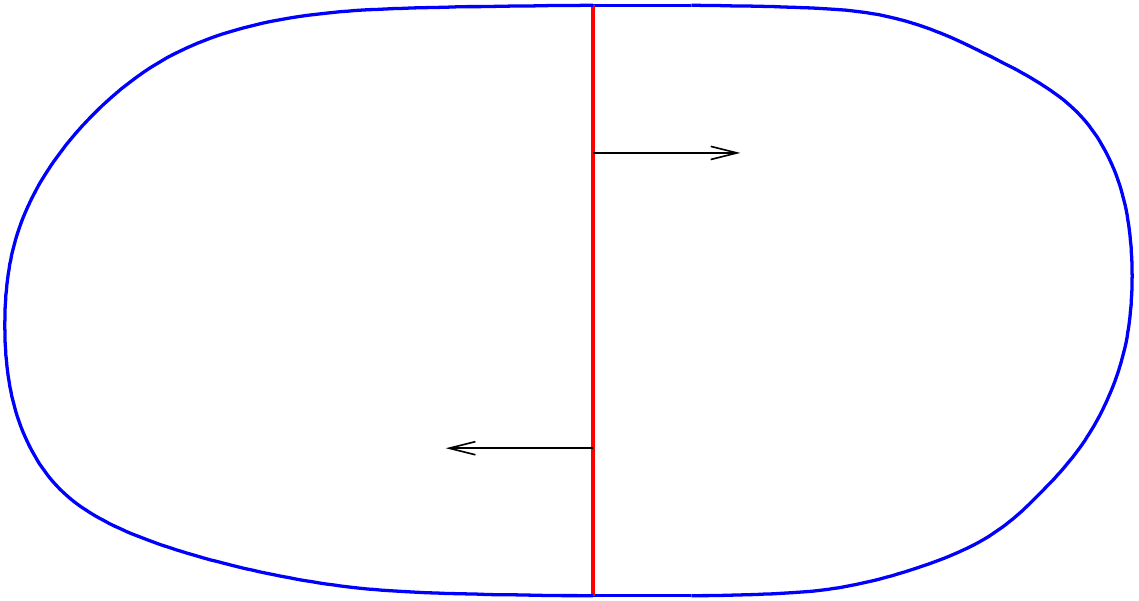} \\}
\put(40,40){$ \Omega_{1} $}
\put(135,40){$ \Omega_{2} $}
\put(115,77){$ \bn_{1}$}
\put(75,15){$ \bn_{2} $}
\put(85,50){$ \gamma$}
\end{picture}
\end{minipage}
\caption{Left: The domain $ \Omega $ with the fracture $ \Omega_{f} $. Right: The domain $ \Omega $ with the interface-fracture $ \gamma $.} 	
\label{Fig:A3TwoModels} 
\end{figure}
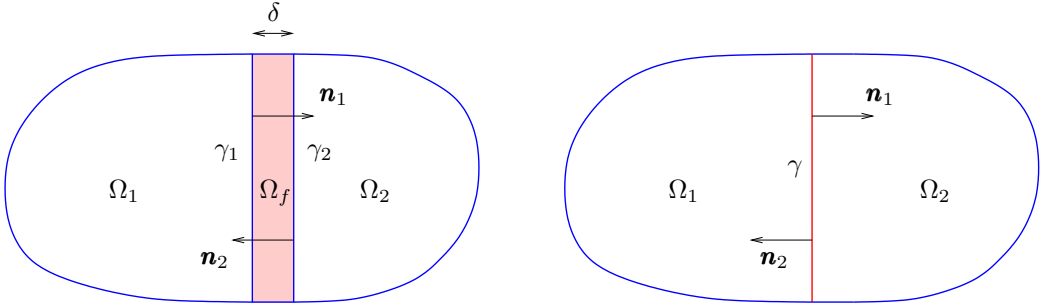
We use the convention that for any scalar, vector or tensor valued function $\phi$ defined on $\Omega$, $\phi_i$ denotes
the restriction of $ \phi$ to $ \Omega_{i}, i=1,2,f $. We rewrite problem \eqref{A3primal} as the following transmission problem:
\begin{equation} \label{A32dprob}
\begin{array}{clll} s_{i}\partial_{t} p_{i} + \Div \bu_{i} &=q_{i} & \text{in} \; \Omega_{i} \times (0,T), & i=1,2,f, \\
\bu_{i} & = -\bK_{i}  \nabla p_{i} & \text{in} \; \Omega_{i} \times (0,T), & i=1,2,f,\\
p_{i} & = 0 & \text{on} \; \left (\partial \Omega_{i} \cap \partial \Omega \right) \times (0,T),& i=1,2,f,\\
p_{i}&=p_{f} & \text{on} \; \gamma_{i} \times (0,T), & i=1,2,\\
\bu_{i} \cdot \bn_{i} &= \bu_{f} \cdot \bn_{i} & \text{on} \; \gamma_{i} \times (0,T), & i=1,2,\\
p_{i}(\cdot, 0) & = p_{0,i} & \text{in} \; \Omega_{i}, & i=1,2,f.
\end{array}\end{equation}

%
%
%
In the reduced fracture model, the fracture $\Omega_f$ is treated as a simple interface $ \gamma $ between subdomains $ \Omega_{1} $ and $ \Omega_{2} $ (see Figure~\ref{Fig:A3TwoModels}, right).
We use the notation $ \nabla_{\tau}$ (respectively $ \Divtau $) for the tangential gradient (respectively tangential divergence) operators along the fracture $ \gamma $. We denote by $ s_{\gamma} $ and $ \bK_{\gamma} $ the storage coefficient and the permeability tensor in the $(d-1)$-dimensional fracture $\gamma$.
  The reduced model that we consider was derived in \cite{Alboin, Vincent}. It may be obtained 
  by averaging across the transversal cross sections of the $d$-dimensional fracture $\Omega_f$. 
It consists of equations in the subdomains,
\begin{equation} \label{A31dprobSub}
\begin{array}{cll} s_{i}\partial_{t} p_{i} + \Div \bu_{i} &=q_{i} & \text{in} \; \Omega_{i} \times (0,T), \\
 \bu_{i} & = -\bK_{i}  \nabla p_{i} & \text{in} \; \Omega_{i} \times (0,T), \\
p_{i} & = 0 & \text{on} \; \left (\partial \Omega_{i} \cap \partial \Omega \right) \times (0,T),\\
p_{i}&=p_{\gamma} & \text{on} \; \gamma \times (0,T),\\
p_{i}(\cdot, 0) & = p_{0,i} & \text{in} \; \Omega_{i}, 
\end{array} \qquad   \text{for} \; i=1,2, \end{equation} 
and equations in the interface fracture
\begin{equation} \label{A31dprobFrac}
\begin{array}{cll} s_{\gamma}\partial_{t} p_{\gamma} + \Divtau  \bu_{\gamma} &= \left (\bu_{1} \cdot \bn_{1\mid \gamma} +  \bu_{2} \cdot \bn_{2 \mid \gamma}\right )  & \text{in} \; \gamma \times (0,T), \vspace{0.1cm}\\
\bu_{\gamma}  & = -\bK_{\gamma} \delta \nabla_{\tau} p_{\gamma} & \text{in} \; \gamma \times (0,T), \\
p_{\gamma} & = 0 & \text{on} \; \partial \gamma \times (0,T),\\
p_{\gamma}(\cdot, 0) & = p_{0,\gamma} & \text{in} \; \gamma.
\end{array}\end{equation} 
These equations are the mass conservation equation and the Darcy equation in the subdomain together with the lower dimensional mass conservation and Darcy equations in the fracture of co-dimension $ 1 $. These two systems are coupled: the fracture sees the subdomain through the source term in the conservation equation in the fracture which represents the difference between the fluid entering the fracture from one subdomain and that exiting through the other subdomain. Each subdomain sees the fracture through the Dirichlet boundary condition imposed on the part of its boundary common with the fracture.
We make the hypothesis of the following compatibility conditions:
$    p_{0,i}=p_{0,\gamma} \mbox{ on } \gamma$,  for $i=1,2$.
For a general mathematical treatment of this type of problem in the stationary case
  see~\cite{Sanchez-Palencia}.

To prove the well-posedness of
problem~\eqref{A31dprobSub}-\eqref{A31dprobFrac}, we shall use the
abstract framework of Subsection~\ref{sec:abstractresult} and apply Theorem~\ref{thm:abst}.  We
first write the weak formulation for
problem~\eqref{A31dprobSub}--~\eqref{A31dprobFrac}, and define the
appropriate function spaces, and the forms on these spaces. We use the
convention that if $V$ is a space of functions, then 
$\pmb{V}$ is a space of vector functions having each component in
$V$. For an arbitrary domain~$ \mathcal{O} $, we denote by $ (\cdot,
\cdot )_{\mathcal{O}} $ the inner product in $ L^{2}(\mathcal{O})$ or
$ \pmb{L^{2}}(\mathcal{O}) $ and by and $ \| \cdot \|_{\mathcal{O}}$
the $ L^{2}(\mathcal{O})$-norm or $\pmb{L^{2}}(\mathcal{O})$-norm. To write the weak formulation of
\eqref{A31dprobSub}-\eqref{A31dprobFrac}, we define the following
Hilbert spaces: 
\begin{align*}
 M&=\left \{ \mu =(\mu_{1}, \mu_{2}, \mu_{\gamma}) \in L^{2}(\Omega_{1}) \times L^{2}(\Omega_{2}) \times L^{2}(\gamma)\right \},\\
\Sigma & =\big \{ \bv = (\bv_{1}, \bv_{2}, \bv_{\gamma}) \in \pmb{L^{2}}(\Omega_{1}) \times \pmb{L^{2}}(\Omega_{2}) \times \pmb{L^{2}}(\gamma): \; \Div \bv_{i} \in L^{2}(\Omega_{i}), i=1,2, \\
& \hspace{3cm} \text{and} \; \Divtau \bv_{\gamma} -\sum_{i=1}^{2} \bv_{i} \cdot \bn_{i \mid \gamma} \in L^{2}(\gamma) \big \}, \vspace{-0.3cm}
\end{align*}
equipped with the norms
\begin{align*}
\| \mu \|^{2}_{M} &= \sum_{i=1}^{2} \| \mu_{i}\|^{2}_{\Omega_{i}} + \| \mu_{\gamma}\|^{2}_{\gamma}, \\
\| \bv \|^{2}_{\Sigma} & = \sum_{i=1}^{2} \left (\| \bv_{i}\|_{\Omega_{i}}^{2} + \| \Div \bv_{i} \|^{2}_{\Omega_{i}} \right ) + \| \bv_{\gamma}\|_{\gamma}^{2}  + \|\Divtau \bv_{\gamma} -\sum_{i=1}^{2} \bv_{i} \cdot \bn_{i\mid \gamma} \|^{2}_{\gamma}.
\end{align*}
We define the following bilinear forms
\begin{equation*}
\begin{array}{lccl}
	a:& \Sigma \times \Sigma &
          \longrightarrow & \R \\[.1cm]
	& (\bu, \bv)& \mapsto & a(\bu, \bv) =\ds \sum_{i=1}^{2} \left (\bK^{-1}_{i} \bu_{i}, \bv_{i}\right )_{\Omega_{i}} + \left ((\bK_{\gamma} \delta)^{-1} \bu_{\gamma}, \bv_{\gamma}\right )_{\gamma},\\
	b:& \Sigma \times M &
          \longrightarrow & \R \\[.1cm]
	& (\bu, \mu)& \mapsto & b(\bu, \mu) = \ds\sum_{i=1}^{2} \left (\Div \bu_{i}, \mu_{i}\right )_{\Omega_{i}} + \left (\Divtau \bu_{\gamma}- \sum_{i=1}^{2} \bu_{i} \cdot \bn_{i\mid \gamma} , \mu_{\gamma}\right )_{\gamma},\\
	c_{s}:& M \times M &
          \longrightarrow & \R \\[.1cm]
	& (\eta,\mu)& \mapsto & c_{s}(\eta, \mu) =\ds \sum_{i=1}^{2} \left (s_{i}\eta_{i}, \mu_{i}\right )_{\Omega_{i}} + \left (s_{\gamma}\eta_{\gamma}, \mu_{\gamma}\right )_{\gamma},\\
\end{array}
\end{equation*}
and the linear form
\begin{equation*}
	\begin{array}{lccl}
	L_{q}:& M &
          \longrightarrow & \R \\[.1cm]
	& \mu & \mapsto & L_{q}(\mu)= \ds\sum_{i=1}^{2} \left (q_{i}, \mu_{i}\right )_{\Omega_{i}}.
	\end{array}
\end{equation*}
With these spaces and forms, the weak form of \eqref{A31dprobSub}-\eqref{A31dprobFrac} can be written as follows:
\begin{eqnarray}
 \mbox{Find  $ p \in H^1(0,T;M) $ and $ \bu  \in L^2(0,T;\Sigma) $ such that,} 
   \nonumber\\
   \begin{array}{rll}
a(\bu, \bv) - b(\bv, p) &=0, & \forall \bv \in \Sigma,\\
c_{s}(\partial_{t} p, \mu) + b (\bu, \mu) & = L_{q} (\mu), & \forall \mu \in M,\\
\end{array}  \hspace{2cm}\label{A31dWeak} \\
\mbox{together with the initial conditions}\hspace{4cm} \nonumber\\
    \begin{array}{rll} p_{i}(\cdot,0) & = p_{0,i} & \text{in} \; \Omega_{i}, \quad i=1,2,\\
  p_{\gamma}(\cdot, 0) & = p_{0,\gamma} & \text{in} \; \gamma,
  \end{array} \hspace{2cm}\label{A31dWeakIC}
\end{eqnarray}
for $p_{0,i} \in L^2(\Omega_i), \  i=1,2$ and $p_{0,\gamma} \in L^2(\gamma)$.
We also define the space 
\begin{eqnarray*}
\HicF :=  \{ \mu =(\mu_{1}, \mu_{2}, \mu_{\gamma}) \in H^{1}(\Omega_{1}) \times H^{1}(\Omega_{2}) \times H^{1}_{0}(\gamma): \mu_{i} = 0 \; \text{on} \; \partial \Omega_{i} \cap \partial \Omega  ,\\
\text{and} \; \mu_i=\mu_{\gamma} \; \text{on} \; \gamma,
\ \; i=1,2 \},
\end{eqnarray*}
equipped with the norm
$$ \| \mu \|_{\HicF} ^{2}= \|\mu\|_{M}^{2} + \sum_{i=1}^{2} \| \nabla \mu_{i}\|_{\Omega_{i}}^{2} + \| \nabla_{\tau} \mu_{\gamma} \|_{\gamma}^{2}.
$$
The well-posedness of problem~\eqref{A31dWeak}-\eqref{A31dWeakIC} is given by the following theorem:
\begin{theorem}\label{A3thrm} 
Assume that there exist four positive constants $ s_{-} $ and $ s_{+} $, $ K_{-} $ and $ K_{+} $ such that \\
\indent \hspace{-4.5mm} $ \bullet$ $ s_{-} \leq s_{i}(x) \leq s_{+} $ for a.e. $x \in \Omega_{i}$, $ i=1,2, $\\
\indent \hspace{-4.5mm} $\bullet$ $ s_{-} \leq s_{\gamma}(x) \leq s_{+} $ for a.e. $x \in \gamma$,\\
\indent \hspace{-4.5mm} $\bullet$ $\varsigma^{T} \bK^{-1}_{i}(x) \varsigma \geq K_{-} \vert \varsigma \vert ^{2}$, 
and $|\bK_{i}(x)\varsigma| \leq K_{+} |\varsigma|$, for a.e. $x \in \Omega_{i}$, $\forall \varsigma \in \R^{d}$, $i=1,2$,\\
\indent \hspace{-3.3mm}$\bullet$ $\eta^{T} (\bK_{\gamma} (x)\delta)^{-1} \eta \geq K_{-} \vert \eta \vert ^{2}$ and $ |(\bK_{\gamma} (x)\delta)^{-1} \eta| \leq K_{+} | \eta |$ for a.e. $x \in \gamma$, $\forall \eta \in \R^{d-1}$.
 If $ q $ is in $ L^{2} (0,T; M) $ and $ p_0=(p_{0,1},p_{0,2},p_{0,\gamma}) $ in $\HicF $
then problem~\eqref{A31dWeak}-\eqref{A31dWeakIC} has a unique solution 
$(p, \bu) \in  \, H^{1}(0,T; M) \times L^{2}(0,T; \Sigma).$
\end{theorem}
%

\begin{proof}
First notice that under the assumptions on $s_i$ and $s_\gamma$ stated in
the theorem, $c_s$ defines an inner product on $M \times M$, and that
the associated norm is equivalent to the original norm on $M$. 

We will apply Theorem~\ref{thm:abst}, in the case $c=0$. The bilinear
forms $a$ and $b$ are obviously continuous, and,
with the hypotheses concerning  $\bK_i$ and $\bK_\gamma$,
$a$ is positive definite on $\Sigma$.

We now check hypothesis~\eqref{eq:H4}. This is easiest to do using the
operator form~\eqref{eq:H4p}, where for this problem $B$ is defined by
\begin{equation*}
  \forall \bu \in \Sigma, \; B \bu = \left( \Div \bu_1, \Div \bu_2,
    \Divtau \bu_{\gamma} - \sum_{i=1}^2 \bu_i \cdot \bn_{i \mid \gamma} \right).
\end{equation*}
The result follows from the hypothesis on $K_i$ and $K_\gamma$ and the
definition of the norm on $\Sigma$.

Last, to check hypothesis~\eqref{eq:H5}, we can take $W=H^1_*$, and use
Green's formula to see that for $\bu \in \Sigma$ and $\mu \in W$
\begin{align*}
b(\bu, \mu) &= \sum_{i=1}^{2} \left (\Div \bu_{i}, \mu_i\right
)_{\Omega_{i}} + \left (\Divtau \bu_{\gamma}- \sum_{i=1}^{2}
  \bu_{i} \cdot \bn_{i\mid \gamma} , \mu_{\gamma}\right )_{\gamma}\\ 
& =  \sum_{i=1}^{2} \left( -\left (\bu_{i}, \nabla \mu_{i}\right
)_{\Omega_{i}} + \left( \bu_i \cdot \bn_{i \mid \gamma}, \mu_i
\right)_\gamma \right)
- \left (\bu_{\gamma}, \nabla_{\tau} \mu_{\gamma}\right )_{\gamma} - \sum_{i=1}^{2} 
\left( \bu_{i} \cdot \bn_{i\mid \gamma} , \mu_{\gamma}\right
)_{\gamma} \\
&=   -\sum_{i=1}^{2} \left (\bu_{i}, \nabla \mu_{i}\right
)_{\Omega_{i}} - \left (\bu_{\gamma}, \nabla_{\tau} \mu_{\gamma}\right )_{\gamma} 
\end{align*}
because $\mu_{i \mid \gamma} = \mu_{\gamma}$.
To conclude, we bound the terms of the right hand side:
\begin{gather*}
   \left (\bu_{i}, \nabla \mu_{i}\right)_{\Omega_{i}}  \leq K_+
  \|\bu_i\|_{\Omega_i} \, \| \mu_i\|_{H^1(\Omega_i)}, \\
 \left (\bu_{\gamma}, \nabla_{\tau} \mu_{\gamma}\right )_{\gamma} \leq
 K_+ \| \bu_\gamma \|_\gamma \| \, \mu_\gamma \|_{H^1(\gamma)}, 
\end{gather*}
from which hypothesis~\eqref{eq:H5} easily follows.
\end{proof}
%
%

It is natural to use domain decomposition methods for
obtaining a numerical solution of
problem \eqref{A32dprob} or problem \eqref{A31dprobSub}-\eqref{A31dprobFrac}, especially as these
methods make it possible to take
different time steps in the subdomains and in the fracture. For problem \eqref{A32dprob}, it
would be a straightforward application of the methods introduced in~\cite{PhuongSINUM} while for problem \eqref{A31dprobSub}-\eqref{A31dprobFrac}, we need to derive a different formulation. In the following, we present two global-in-time domain decomposition methods for solving \eqref{A31dprobSub}-\eqref{A31dprobFrac} based on different transmission conditions. A space-time interface problem, which will be solved iteratively, is derived for each approach.
%
%
%
%
\section{Global-in-time preconditioned Schur (GTP-Schur): using the time-dependent Steklov-Poincar\'e operator}
\label{A3Sub:SchurM1} 
The Global-in-time preconditioned Schur (GTP-Schur) method is directly derived from the formulation of problem \eqref{A31dprobSub} - \eqref{A31dprobFrac}.
To obtain the interface problem for this method, we need to introduce some notation.
For a bounded domain $ \iO \in \mR^{d} \; (d=2,3) $ with Lipschitz boundary
$ \partial \iO$
containing an open subset $\gamma \subset \partial\iO$, we define
the space
\begin{align*}
H_{*,\gamma}^{1}(\iO)
   &:=\left \{ \mu \in H^{1}(\iO): \  \mu = 0 \; \; \text{on} \; \left (\partial \iO \setminus \gamma\right ) 
               \right \}.            
\end{align*}

Then we define the following Dirichlet to Neumann operators $ \iS^{\text{DtN}}_{i} $, $ i=1,2: $
\begin{equation*} 
\begin{array}{rl} \iS^{\text{DtN}}_{i}: H^{1}(0,T; H^{1\over 2}_{00}(\gamma)) \times L^{2} (0,T;L^{2}(\Omega_{i})) \times H_{*,\gamma}^{1} (\Omega_i) &\rightarrow L^{2}\left (0,T; (H^{1\over 2}_{00}(\gamma))^\prime\right ) \vspace{0.1cm}\\
\iS^{\text{DtN}}_{i}(\lambda, q, p_{0}) &\mapsto \bu_{i} \cdot \bn_{i \mid \gamma},
\end{array}
\end{equation*}
where $ (p_{i}, \bu_{i}), \; i=1,2, $ is the solution of the problem
\begin{equation} \label{A3subprobM1}
\begin{array}{rll} s_{i}\partial _{t} p_{i}+\Div \bu_{i} & = q & \text{in} \; \Omega_{i} \times (0,T) , \\
\bu_{i} & = -\bK_{i}\nabla p_{i} & \text{in} \; \Omega_{i} \times (0,T),\\
p_{i} &=0 & \text{on} \; \left (\partial \Omega_{i} \cap \partial \Omega \right) \times (0,T), \\
 p_{i} & = \lambda & \text{on} \; \gamma \times (0,T),\\
p_{i}(\cdot,0) & = p_{0} & \text{in} \; \Omega_{i}.
\end{array}
\end{equation}
\begin{remark} 
A straightforward application of Theorem~\ref{thm:abst} shows the
well-posedness of subdomain problem \eqref{A3subprobM1}.
See also~\cite{PhuongThesis,Arbogast} for a direct proof.
\end{remark} 

Problem~\eqref{A31dprobFrac} is reduced to an interface problem with unknowns $ \lambda $ and $u_\gamma$:
\begin{equation} \label{A3finterfaceM1-short}
\begin{array}{rll} 
s_{\gamma}\partial_{t} \lambda +\Divtau \bu_{\gamma} &=\sum_{i=1}^{2}\iS^{\text{DtN}}_{i}(\lambda, q_{i}, p_{0,i}) & \text{in} \; \gamma \times (0,T), \\
\bu_{\gamma}& = -\bK_{\gamma} \delta \nabla_{\tau} \lambda & \text{in} \; \gamma \times (0,T), \\
\lambda & = 0 & \text{on} \; \partial \gamma \times (0,T),\\
\lambda (\cdot, 0) & = p_{0,\gamma} & \text{in} \; \gamma.
\end{array}
\end{equation}
or equivalently
\begin{equation} \label{A3finterfaceM1}
\begin{array}{rll} 
s_{\gamma}\partial_{t} \lambda +\Divtau \bu_{\gamma} -\sum_{i=1}^{2}\iS^{\text{DtN}}_{i}(\lambda, 0, 0) &=\sum_{i=1}^{2}\iS^{\text{DtN}}_{i}(0, q_{i}, p_{0,i})& \text{in} \; \gamma \times (0,T), \\
\bu_{\gamma}& = -\bK_{\gamma} \delta \nabla_{\tau} \lambda & \text{in} \; \gamma \times (0,T), \\
\lambda & = 0 & \text{on} \; \partial \gamma \times (0,T),\\
\lambda (\cdot, 0) & = p_{0,\gamma} & \text{in} \; \gamma,
\end{array}
\end{equation}
or in compact form (space-time),
\begin{equation*} 
\iS \left(\begin{array}{l} 
\lambda\\
\bu_{\gamma}
\end{array}\right)
 = \chi.
\end{equation*}
This problem is solved using an iterative solver such as GMRES since due to the time derivative
the system is nonsymmetric.

To improve the convergence of the iterative algorithm, we will consider two preconditioners. The first,
introduced in~\cite{Laila}, arises from the observation that the interface problem is dominated by the second order operator $ \left (\Divtau (\bK_{\gamma} \delta \nabla_{\tau})\right) $ since the Steklov-Poincar\'e operator is of lower order (first order). This is even more the case when the permeability in the fracture is much larger than that
  in the surrounding domain. Thus one choice for a preconditioner is $ \bP_{\text{loc}}^{-1} $ defined by taking the discrete counterpart of the operator $ \left (\Divtau (\bK_{\gamma} \delta \nabla_{\tau})\right) ^{-1} $.
We have
\begin{equation*} 
\begin{array}{rl} \bP_{\text{loc}}^{-1}:  L^{2}(\gamma) &\rightarrow L^{2}(\gamma) \vspace{0.1cm}\\
 g_{\gamma} &\mapsto \tilde{p}_{\gamma},
\end{array}
\end{equation*}
where $ (\tilde{p}_{\gamma}, \tilde{\bu}_{\gamma}) $ is the solution of the problem
\begin{equation*} 
\begin{array}{rll} 
\Divtau \tilde{\bu}_{\gamma} &=g_{\gamma} & \text{in} \; \gamma, \\
\tilde{\bu}_{\gamma}& = -\bK_{\gamma} \delta \nabla_{\tau} \tilde{p}_{\gamma} & \text{in} \; \gamma, \\
\tilde{p}_{\gamma} & = 0 & \text{on} \; \partial \gamma.
\end{array}
\end{equation*}
This preconditioner was introduced for elliptic problems, and it was shown numerically~\cite{Laila} that
it significantly improves the convergence of the algorithm, especially, as mentionned before, for high
permeability in the fracture.

A second possibility is to use the Neumann-Neumann preconditioner as was done in~\cite{PhuongThesis,PhuongSINUM} for ordinary domain decomposition algorithms (i.e. without fractures). The preconditioned problem is then
$$ \bP_{\text{NN}}^{-1} \varphi = \tilde{\chi}, $$
with
$$ \bP_{\text{NN}}^{-1} := \left (\sigma_{1}(\check{\iS}^{\text{DtN}}_{1})^{-1} +\sigma_{2} (\check{\iS}^{\text{DtN}}_{2})^{-1}\right ),
$$
where
$ \sigma_{i}: \Gamma \times (0,T) \rightarrow [0,1] $ is such that $ \sigma_{1}+\sigma_{2} = 1$. If $ \bK_{i} = \mathfrak{K}_{i} \bI $ and $ \mathfrak{K}_{i} $ is constant in each subdomain then
	$$ \sigma_{i} = \frac{\mathfrak{K}_{i} }{\mathfrak{K}_{1} + \mathfrak{K}_{2}}.
	$$
The operator $ (\check{\iS}^{\text{DtN}}_{i})^{-1}, \; i=1,2,$ is the inverse of the operator
$
\check{\iS}^{\text{DtN}}_{i}:= \iS^{\text{DtN}}_{i}(\cdot, 0, 0),
$
and is defined by
\begin{equation*} 
\begin{array}{rl} (\check{\iS}^{\text{DtN}}_{i})^{-1}:  L^{2}\left (0,T; L^{2}(\gamma)\right ) &\rightarrow H^{1}\left (0,T; L^{2}(\gamma)\right ) \vspace{0.1cm}\\
\left (\check{\iS}^{\text{DtN}}_{i}\right )^{-1} (\varphi) &\mapsto p_{i \mid \gamma},
\end{array}
\end{equation*}
where $ (p_{i}, \bu_{i}), \; i=1,2, $ is the solution of the problem
\begin{equation} \label{A3subprobM1-2}
\begin{array}{rll} s_{i}\partial _{t} p_{i}+\Div \bu_{i} & = 0 & \text{in} \; \Omega_{i} \times (0,T) , \\
\bu_{i} & = -\bK_{i}\nabla p_{i} & \text{in} \; \Omega_{i} \times (0,T),\\
p_{i} &=0 & \text{on} \; \left (\partial \Omega_{i} \cap \partial \Omega \right) \times (0,T), \\
-\bu_{i} \cdot \bn_{i} & = \varphi & \text{on} \; \gamma \times (0,T),\\
p_{i}(\cdot,0) & = 0 & \text{in} \; \Omega_{i}.
\end{array}
\end{equation}
In Section~\ref{A3Sec:Num}, we will carry out numerical experiments and compare the performance of these two preconditioners. 
%
%
%
%
%
\section{Global-in-time optimized Schwarz (GTO-Schwarz): using optimized Schwarz waveform relaxation}
\label{A3subsec:VentcellDiff}
While the extension of the GTP-Schur method to handle the fracture model is straightforward, the extension
of the GTO-Schwarz method to the fracture problem needs something more. Indeed,
instead of imposing Dirichlet boundary conditions on $ \gamma \times (0,T) $ when solving the fracture problem
as was done for the GTP-Schur method, for the GTO-Schwarz approach one uses optimized Robin transmission
conditions. Thus, we introduce new transmission conditions, that combine the equation for continuity of the pressure across the fracture with the flow equations \eqref{A31dprobFrac} in the fracture. These new transmission conditions contain a free parameter, which is used to accelerate the convergence. This is an extension of the OSWR method with optimized Robin parameters studied in~\cite{PhuongThesis,PhuongSINUM} in which Robin-to-Robin transmission conditions are considered in mixed form. Here however, because of the fracture problem, we obtain what we will call Ventcell-to-Robin transmission conditions as described below.
\subsection{Ventcell-to-Robin transmission conditions}
The new transmission conditions are derived by introducing Lagrange multipliers $ p_{i,\gamma}, \; i=1,2, $ with
$p_{i,\gamma}$ representing the trace on the interface $ \gamma $ of the pressure $ p_{i} $ in the subdomain
$\Omega_i$. As the pressure is continuous across the interface, one has
\begin{equation} \label{A4OSpressure}
p_{1,\gamma} = p_{2, \gamma} = p_{\gamma}, \quad \text{on} \; \gamma \times (0,T).
\end{equation}
We then rewrite the Darcy equation in the fracture associated with each $ p_{i, \gamma} $ as 
$$ \bu_{\gamma, i} = -\bK_{\gamma} \delta \nabla_{\tau} p_{i, \gamma}, \quad \text{on} \; \gamma \times (0,T), \; i=1,2.
$$
We have used the notation $ \bu_{\gamma,i}, \; i=1,2, $ instead of $ \bu_{i, \gamma} $ to insist on the fact that $ \bu_{\gamma,i} $ is not the tangential component of a trace of $ \bu_{i} $ on $ \gamma $.
In fact, $\bu_{\gamma,i}, \; i=1,2, $ represents the tangential velocity in the fracture associated with the pressure
$p_i$ so that
$$ \bu_{\gamma,1} = \bu_{\gamma,2} = \bu_{\gamma}, \quad \text{on} \; \gamma \times (0,T), \; i=1,2.
$$
With the notation introduced above, the flow equation \eqref{A31dprobFrac} in the fracture can be rewritten, for $ i=1,2, $ and $ j=(3-i), $ as
\begin{equation} \label{A3OSfractureprob}
\begin{array}{rll} -\bu_{i} \cdot \bn_{i} + s_{\gamma}\partial _{t} p_{i, \gamma} +\Divtau \bu_{\gamma, i} &= -\bu_{j} \cdot \bn_{i}, &\text{on} \; \gamma \times (0,T), \vspace{0.1cm}\\
\bu_{\gamma, i}& =-\bK_{\gamma} \delta \nabla_{\tau} p_{i, \gamma}, &\text{on} \; \gamma \times (0,T),  \vspace{0.1cm} \\
p_{i, \gamma} &= 0 & \text{on} \; \partial \gamma \times (0,T),\\
p_{i, \gamma}(\cdot, 0) &= p_{0,\gamma} & \text{in} \; \gamma.
\end{array} 
\end{equation}
In the context of domain decomposition, \eqref{A4OSpressure} and
\eqref{A3OSfractureprob} are the coupling conditions between the
subdomains. As in the case without a fracture we take a linear
combination of these conditions (for a parameter $\alpha >0$),
but here we obtain equivalent Ventcell-to-Robin transmission conditions
(instead of  Robin-to-Robin): 
\begin{equation} \label{A3TCs1M2}
\begin{array}{rl}  -\bu_{1} \cdot \bn_{1} + \alpha p_{1, \gamma} + s_{\gamma}\partial _{t} p_{1, \gamma} +\Divtau \bu_{\gamma, 1} &= -\bu_{2} \cdot \bn_{1} + \alpha p_{2, \gamma} \vspace{0.1cm}\\
\bu_{\gamma, 1}& =-\bK_{\gamma} \delta \nabla_{\tau} p_{1, \gamma}  \vspace{0.1cm}
\end{array} \qquad \text{on} \; \gamma \times (0,T), \vspace{-0.3cm}
\end{equation}
\begin{equation} \label{A3TCs2M2}
\begin{array}{rl}
-\bu_{2} \cdot \bn_{2} + \alpha p_{2, \gamma} + s_{\gamma}\partial _{t} p_{2, \gamma} +\Divtau \bu_{\gamma, 2} &= -\bu_{1} \cdot \bn_{2} + \alpha p_{1, \gamma}  \vspace{0.1cm}\\
\bu_{\gamma, 2}& =-\bK_{\gamma} \delta \nabla_{\tau} p_{2, \gamma} 
\end{array} \qquad \text{on} \; \gamma \times (0,T),
\end{equation}
Using these transmission conditions together with the boundary and initial conditions 
\begin{equation} \label{A3BCsfM2}
\begin{array}{rlll} p_{1, \gamma} &= p_{2, \gamma} &= 0 & \text{on} \; \partial \gamma \times (0,T),\\
p_{1, \gamma}(\cdot, 0) &= p_{2, \gamma}(\cdot, 0) &= p_{0,\gamma} & \text{in} \; \gamma,
\end{array}
\end{equation}
the subdomain problem is obtained by imposing Ventcell boundary conditions on $ \gamma \times (0,T) $, $ i=1,2, $ $ j=3-i $:
\begin{equation} \label{A3subprobM2}
\begin{array}{rll} s_{i}\partial _{t} p_{i}+\Div \bu_{i} & = q & \text{in} \; \Omega_{i} \times (0,T) , \\
\bu_{i} & = -\bK_{i}\nabla p_{i} & \text{in} \; \Omega_{i} \times (0,T),\\
-\bu_{i} \cdot \bn_{i} + \alpha p_{i, \gamma} +s_{\gamma}\partial _{t} p_{i, \gamma} +\Divtau \bu_{\gamma, i}  & = -\bu_{j} \cdot \bn_{i} + \alpha p_{j, \gamma}  & \text{on} \; \gamma \times (0,T),\\
\bu_{\gamma, i}& = -\bK_{\gamma} \delta \nabla_{\tau} p_{i, \gamma} & \text{in} \; \gamma \times (0,T), \\
p_{i} &=0 & \text{on} \; \left (\partial \Omega_{i} \cap \partial \Omega \right) \times (0,T), \\
p_{i, \gamma} &= 0 & \text{on} \; \partial \gamma \times (0,T),\\
p_{i}(\cdot,0) & = p_{0} & \text{in} \; \Omega_{i}, \\
p_{i, \gamma}(\cdot, 0) & = p_{0,\gamma} & \text{in} \; \gamma,
\end{array}
\end{equation}
where the quantity on the right hand side of the third equation
will be known in the context of an iterative method for solving \eqref{A31dprobSub}-\eqref{A31dprobFrac}.
In the next subsection we prove that problem~\eqref{A3subprobM2} is well-posed. 
%
\subsection{Well-posedness of the subdomain problem with Ventcell boundary conditions}
For a bounded domain $ \iO \subset \mR^{d} \; (d=2,3) $ with Lipschitz boundary
$ \partial \iO$ containing an open subset $\gamma \subset \partial\iO$, consider the following time-dependent problem written in mixed form with Dirichlet and Ventcell boundary conditions
\begin{equation} \label{A3subprobM2-short} 
\begin{array}{rll} s_{\iO}\partial _{t} p_{\iO} +\Div \bu_{\iO} & = q & \text{in} \; \iO \times (0,T) , \\
\bu_{\iO} & = -\bK_{\iO} \nabla p_{\iO} & \text{in} \; \iO \times (0,T),\\
-\bu_{\iO} \cdot \bn + \alpha p_{\gamma} +s_{\gamma}\partial _{t} p_{\gamma} +\Divtau \bu_{\gamma}  & = \theta_\gamma 
& \text{on} \; \partial \gamma \times (0,T),\\
\bu_{\gamma}& = -\bK_{\gamma} \delta \nabla_{\tau} p_{\gamma} & \text{in} \; \gamma \times (0,T), \\
p_{\iO} &=0 & \text{on} \; \left (\partial \iO \setminus \gamma\right ) \times (0,T), \\
p_{\gamma} &= 0 & \text{on} \; \partial \gamma \times (0,T),\\
p_{\iO}(\cdot,0) & = p_{0,{\cal O}} & \text{in} \; \Omega_{i}, \\
p_{\gamma}(\cdot, 0) & = p_{0,\gamma} & \text{in} \; \gamma,
\end{array}
\end{equation}
where $ \theta_\gamma $ is a function defined on $ \gamma \times (0,T) $, and $\alpha \in \R, \ \alpha >0$.
In order to write the weak formulation of~\eqref{A3subprobM2-short}, we need to define the following Hilbert spaces: 
\begin{align*}
\MT &=\left \{ \mu = (\mu_{\iO}, \mu_{\gamma}) \in L^{2}(\iO) \times  L^{2}(\gamma)\right \},\\
\ST & =\left \{ \bv=(\bv_{\iO},\bv_{\gamma}) \in \pmb{L^{2}}(\iO) \times \pmb{L^{2}}(\gamma): \Div \bv_{\iO} \in H(\Div, \iO) \; \text{and} \; \left (\Divtau \bv_{\gamma} - \bv_{\iO} \cdot \bn_{\mid \gamma}\right ) \in L^{2}(\gamma)\right \},
\end{align*}
equipped with the norms
\begin{align*}
\| \mu \|^{2}_{\MT} &=  \| \mu_{\iO}\|^{2}_{\iO} + \| \mu_{\gamma}\|^{2}_{\gamma}, \\
\| \bv \|^{2}_{\ST} & = \| \bv_{\iO} \|_{\iO} + \| \Div \bv_{\iO} \|^{2}_{\iO}  + \| \bv_{\gamma} \|_{\gamma}^{2}+ \|\Divtau \bv_{\gamma}-\bv_{\iO} \cdot \bn_{\mid \gamma}\|_{\gamma}^{2}.
\end{align*}
Then define the bilinear forms
\begin{equation*}
\begin{array}{lcclcrl}
	a_{\cal O}:& \ST \times \ST &
          \longrightarrow & \R, & a_{\cal O}(\bu, \bv) &=& \left (\bK^{-1}_{\iO} \bu_{\iO},
    \bv_{\iO}\right )_{\iO} + \left ((\bK_{\gamma} \delta)^{-1}
    \bu_{\gamma}, \bv_{\gamma}\right )_{\gamma} \\
	b_{\cal O}:& \ST \times \MT &
          \longrightarrow & \R, & b_{\cal O}(\bu, \mu) &=& \left (\Div \bu_{\iO}, \mu_{\iO}\right )_{\iO} + \left (\Divtau \bu_{\gamma}-\bu_{\iO} \cdot \bn_{\mid \gamma }, \mu_{\gamma}\right )_{\gamma},\\
	c_{\cal O}:& \MT \times \MT &
          \longrightarrow & \R, & c_{\cal O}(\eta, \mu) &=& \left (\alpha \eta_{\gamma}, \mu_{\gamma}\right )_{\gamma},\\
	c_{s,{\cal O}}:& \MT \times \MT &
          \longrightarrow & \R,  & c_{s,{\cal O}}(\eta, \mu) &=& \left (s_{\iO}\eta_{\iO}, \mu_{\iO}\right )_{\iO} + \left (s_{\gamma}\eta_{\gamma}, \mu_{\gamma}\right )_{\gamma},\\
\end{array}
\end{equation*}
and the linear form
\begin{equation*}
	\begin{array}{lcclc}
	L_{q,{\cal O}}:& \MT &
          \longrightarrow & \R, & L_{q,{\cal O}}(\mu)=  \left (q, \mu_{\iO}\right )_{\iO} + \left (\theta_\gamma 
          , \mu_{\gamma}\right )_{\gamma}.
	\end{array}
\end{equation*}
With these spaces and forms, the weak form of \eqref{A3subprobM2-short} can be written as follows:
\begin{eqnarray} 
\mbox{For a.e. $ t \in (0,T) $, find $ p(t) \in \MT $ and $ \bu (t) \in \ST $ such that}
   \hspace{4cm}\nonumber\\  
\begin{array}{rll} 
a_{\cal O}(\bu, \bv) - b_{\cal O}(\bv, p) &=0 & \forall \bv \in \ST,\\
c_{s,{\cal O}}(\partial_{t} p, \mu) +c_{\cal O}(p,\mu)+ b_{\cal O}(\bu, \mu) & = L_{q,{\cal O}} (\mu) & \forall \mu \in \MT,\\
\end{array} \hspace{2cm} \label{A3WeaksubprobM2}\\
\mbox{together with the initial conditions} \hspace{4.5cm} \nonumber \\
  \begin{array}{rll} p_{\cal O}(\cdot,0) & = p_{0,{\cal O}} & \text{in} \; \iO,\\
  p_{\gamma}(\cdot, 0) & = p_{0,\gamma} & \text{in} \; \gamma.
  \end{array}\hspace{4cm} \label{A3WeaksubprobM2IC} 
\end{eqnarray}
We will also make use of
\begin{align*}
\Hoo(\iO,\gamma) &:=   H_{*,\gamma}^{1} (\iO)\times H^1_0(\gamma).            
\end{align*}
The well-posedness of problem~\eqref{A3WeaksubprobM2}-\eqref{A3WeaksubprobM2IC} is given by the following theorem:
\begin{theorem}\label{A3thrmVentcell} 
Assume that there exist positive constants $ s_{-} $, $ s_{+} $, $ K_{-} $, $ K_{+} $ with
 \begin{itemize}
\item $ s_{-} \leq s_{\iO}(x) \leq s_{+} $ for a.e. $x \in \iO$,
\item $ s_{-} \leq s_{\gamma}(x) \leq s_{+} $ for a.e. $x \in \gamma$,
\item $\varsigma^{T} \bK^{-1}_{\iO}(x) \varsigma \geq K_{-} \vert \varsigma \vert ^{2}$, 
and $|\bK_{i}(x)\varsigma| \le K_{+} |\varsigma|$, for a.e. $x \in \iO$, $\forall \varsigma \in \R^{d}$, 
\item $\eta^{T} \bK^{-1}_{\gamma}(x) \delta \eta \geq K_{-} \vert \eta \vert ^{2}$ and $ |(\bK_{\gamma} (x)\delta)^{-1} \eta | \leq K_{+} $ for a.e. $x \in \gamma$, $\forall \eta \in \R^{d-1}$.
\end{itemize}
If $ q $ is in $ L^{2} (0,T; \MT) $, $ p_{0}  $ in $\Hoo(\iO,\gamma) $ and $ \theta_\gamma $ in $ L^{2}(0,T; L^{2}(\gamma)) $
then problem \eqref{A3WeaksubprobM2}-\eqref{A3WeaksubprobM2IC} has a unique solution  
$(p, \bu) \in  \, H^{1}(0,T; \MT) \times L^{2}(0,T; \ST).$
\end{theorem} 

\begin{proof}
 As in the case with Dirichlet boundary conditions, we apply
 Theorem~\ref{thm:abst}. Again notice
 that, under the hypotheses on $s_{\cal{O}}$ and $s_\gamma$, $c_{s,
   \cal{O}}$ defines an inner product on $M$, equivalent to its usual
 inner product. 

$\bullet$ It is clear that $a$, $b$ and $c$ are all continuous forms,
that $a$ is positive definite (due to the hypotheses on
$K_{\cal{O}}$ and $K_{\gamma}$)  and that $c_{\cal{O}}$ is positive (since $\alpha \ge 0$), so that
hypotheses~\eqref{eq:H2} and \eqref{eq:H3} in Theorem~\ref{thm:abst} hold. 

$\bullet$ To verify hypothesis~\eqref{eq:H4}, we define the operator $B$
by
\begin{equation*}
B \bu =  \left (\Div \bu_{\iO}, \Divtau \bu_{\gamma}-\bu_{\iO} \cdot
  \bn_{\mid \gamma } \right),
\end{equation*}
and it follows easily that
\begin{multline*}
  \| B \bu \|_M^2 + a_{\cal O}(\bu, \bu) = \| \Div \bu_{\iO} \|_{\cal
    O}^2 + \|\Divtau \bu_{\gamma}-\bu_{\iO} \cdot
  \bn_{\mid \gamma } \|_\gamma^2  \\
+ \left (\bK^{-1}_{\iO} \bu_{\iO},
    \bu_{\iO}\right )_{\iO} + \left ((\bK_{\gamma} \delta)^{-1}
    \bu_{\gamma}, \bu_{\gamma}\right )_{\gamma} \ge \beta \| \bu \|^2_\Sigma,
\end{multline*}
for some $\beta >0$, again because of the lower bounds on $K_{\cal O}$
and $K_\gamma$. 

$\bullet$ Last, to check~\eqref{eq:H5}, we proceed as in
Theorem~\ref{A3thrm}, and use Green's formula for $\bu \in \Sigma$ and
$\mu \in H^{1,1}_*(\cal{O}, \gamma)$. 
\begin{align*}
b_{\cal O}(\bu, \mu) &= \left (\Div \bu_{\iO}, \mu_{\iO}\right )_{\iO}
+ \left (\Divtau \bu_{\gamma}-\bu_{\iO} \cdot \bn_{\mid \gamma },
  \mu_{\gamma}\right )_{\gamma} \\
&= -\left(\bu_{\iO}, \nabla \mu_{\iO} \right)_{\iO} + \left( \bu_{\iO}
\cdot \bn_{\mid \gamma}, \mu_{\iO}  \right)_\gamma - \left(
\bu_\gamma, \nabla_\tau \mu_\gamma \right)_{\gamma} - \left( \bu_{\iO} \cdot \bn_{\mid \gamma },
  \mu_{\gamma}\right )_{\gamma},
\end{align*}
from which the proof of the theorem follows.
\end{proof}

\subsection{The interface problem}
\label{A3Sub:IFm2ConvFact}
As for the GTP-Schur method, we derive an interface problem which in this case is associated with Ventcell-to-Robin transmission conditions~\eqref{A3TCs1M2}-\eqref{A3TCs2M2}.
Towards this end, we define the following Ventcell-to-Robin operator $ \iS^{\text{VtR}}_{i} $, which depends on the parameter $ \alpha $, for $ i=1,2 $; $ j=(3-i) $:
\begin{equation*} 
\begin{array}{rcl} \iS^{\text{VtR}}_{i}:  L^{2}(0,T; L^{2}(\gamma)) \times L^{2} (0,T;L^{2}(\Omega_{i})) \times H_{*}^{1} (\Omega_i)  \times  H_{0}^{1}(\gamma) 
&\rightarrow& L^{2}(0,T; L^{2}(\gamma)) \hspace{6mm}\\
\hspace{2cm}\iS^{\text{VtR}}_{i}(\theta_\gamma, q, p_{0}, p_{0,\gamma}) &\mapsto &-\bu_{i} \cdot \bn_{j \mid \gamma} + \alpha p_{i, \gamma},
\end{array}
\end{equation*}
where $ (p_{i}, \bu_{i}, p_{i,\gamma}, \bu_{\gamma, i}) $ is the solution of the subdomain problem with Ventcell boundary conditions
\begin{equation} \label{A3subprobM2-2}
\begin{array}{rll} s_{i}\partial _{t} p_{i}+\Div \bu_{i} & = q & \text{in} \; \Omega_{i} \times (0,T) , \\
\bu_{i} & = -\bK_{i}\nabla p_{i} & \text{in} \; \Omega_{i} \times (0,T),\\
-\bu_{i} \cdot \bn_{i} + \alpha p_{i, \gamma} +s_{\gamma}\partial _{t} p_{i, \gamma} +\Divtau \bu_{\gamma, i}  & =\theta_\gamma
& \text{on} \; \gamma \times (0,T),\\
\bu_{\gamma, i}& = -\bK_{\gamma} \delta \nabla_{\tau} p_{i, \gamma} & \text{in} \; \gamma \times (0,T), \\
p_{i} &=0 & \text{on} \; \left (\partial \Omega_{i} \cap \partial \Omega \right) \times (0,T), \\
p_{i, \gamma} &= 0 & \text{on} \; \partial \gamma \times (0,T),\\
p_{i}(\cdot,0) & = p_{0} & \text{in} \; \Omega_{i}, \\
p_{i, \gamma}(\cdot, 0) & = p_{0,\gamma} & \text{in} \; \gamma.
\end{array}
\end{equation}

The interface problem with two Lagrange multipliers is then
\begin{equation} \label{A3finterfaceM3}
\begin{array}{rl} 
\theta_{\gamma,1} &= \iS^{\text{VtR}}_{2}(\theta_{\gamma,2}, q_{2}, p_{0,2},  p_{0,\gamma}) \vspace{0.12cm}\\
\theta_{\gamma,2} & =  \iS^{\text{VtR}}_{1}(\theta_{\gamma,1}, q_{1}, p_{0,1}, p_{0,\gamma})
\end{array} \qquad \text{on} \; \gamma \times (0,T),
\end{equation}
or equivalently
\begin{equation} \label{A3finterfaceM3-linear}
\begin{array}{rl} 
\theta_{\gamma,1} -\iS^{\text{VtR}}_{2}(\theta_{\gamma,2}, 0,0,0) & =\iS^{\text{VtR}}_{2}(0, q_{2}, p_{0,2},  p_{0,\gamma}) \vspace{0.12cm}\\
\theta_{\gamma,2} - \iS^{\text{VtR}}_{1}(\theta_{\gamma,1} ,0,0,0) & =\iS^{\text{VtR}}_{1}(0, q_{1}, p_{0,1},  p_{0,\gamma}) 
\end{array} \qquad \text{on} \; \gamma \times (0,T).
\end{equation}
The discrete counterpart of this problem can be solved iteratively using Jacobi iterations or GMRES.
The former choice yields an algorithm equivalent to the OSWR algorithm for the reduced fracture model~\eqref{A31dprobSub} - \eqref{A31dprobFrac} and is written as follows: starting with an initial guess
$ \theta_{\gamma,j}^{0} $, $j=3-i$, on $ \gamma \times (0,T) $ for the first iteration,
$$ -\bu_{i}^{0} \cdot \bn_{i} + \alpha  \; p_{i, \gamma}^{0} + s_{\gamma} \partial _{t} p^{0}_{i, \gamma}
+\Divtau \bu^{0}_{\gamma, i} = \theta_{\gamma,j}^{0},
$$
then at the $ k^{th} $ iteration, $ k=1, \hdots,  $ solve in each subdomain the time-dependent problem, for $ i=1,2; $ $ j=(3-i), $
\begin{equation} \label{A3FracSchwarzalgMixed}
\hspace{-0.5cm}\begin{array}{rll} s_{i} \partial _{t} p^{k}_{i}+\Div \bu^{k}_{i}& = q_{i} & \text{in} \; \Omega_{i} \times (0,T) , \\
\bu^{k}_{i}& =-\bK_{i}\nabla p^{k}_{i} & \text{in} \; \Omega_{i} \times (0,T) , \\
- \bu^{k}_{i} \cdot \bn_{i} + \alpha p^{k}_{i, \gamma} +s_{\gamma} \partial _{t} p^{k}_{i, \gamma} +\Divtau \bu^{k}_{\gamma, i} & =
\theta_{\gamma,j}^{k-1}  & \text{on} \; \gamma \times (0,T),\\
\bu^{k}_{\gamma, i} & =  -\bK_{f,\tau} \delta \nabla_{\tau} p^{k}_{i, \gamma} & \text{on} \; \gamma \times (0,T),\\
p^{k}_{i} & = 0 & \text{on} \; \left (\partial \Omega_{i} \cap \partial \Omega\right ) \times (0,T),\\
p^{k}_{i, \gamma} & = 0 & \text{on} \; \partial \gamma \times (0,T),\\
p^{k}_{i}(\cdot,0) & = p_{0, i} & \text{in} \; \Omega_{i}, \\
p^{k}_{i, \gamma}(\cdot,0) & = p_{0, \gamma} & \text{in} \; \gamma,
\end{array} 
\end{equation}
with $\theta_{\gamma,j}^{k-1} =- \bu^{k-1}_{j} \cdot \bn_{i} + \alpha p^{k-1}_{j, \gamma} $
on $\gamma \times (0,T)$.

The convergence of algorithm~\eqref{A3FracSchwarzalgMixed} depends on the choice of the parameter~$ \alpha $. Thus we extend the analysis for the convergence factor of the OSWR algorithm derived in the case without fractures
\cite{Bennequin,Gander2006,JaphetDD9} to this algorithm and from that, one can calculate the optimal or
optimized values of the parameter $ \alpha $.

\subsection{Convergence factor formula for computing the optimized parameter}
\label{A3Sub:ConvFact}
In this section, we extend the two domain analysis~\cite{Bennequin,OSWR3sub,PhuongThesis, JaphetDD9,VMartin} to derive the convergence factor of the OSWR algorithm introduced in Section~\ref{A3Sub:IFm2ConvFact} for a reduced fracture model for compressible flow. Towards this end, we consider the two half-space decomposition
$ \Omega_{-} =\R ^{-} \times \R,\quad \Omega_{+} = \R^{+} \times \R$
and write the OSWR algorithm, applied to the fractured model, in the primal formulation: at the $ k^{\text{th}} $ Jacobi iteration, solve
\begin{equation} \label{Append:A3FracSchwarzalg1}
\hspace{-0.2cm}\begin{array}{rll} s_{-}\partial _{t} p^{k}_{-}+\Div (-\bK_{-}\nabla p^{k}_{-})& = q & \text{in} \; \Omega_{-} \times (0,T) , \\
\ds\bK_{-} \frac{\partial p^{k}_{-}}{\partial \bn_{-}} + \alpha p^{k}_{-} +s_{\gamma}\partial _{t} p^{k}_{-} +\Divtau (-\bK_{f, \tau} \delta \nabla_{\tau} p^{k}_{-})  & =\ds \bK_{+} \frac{\partial p^{k-1}_{+}}{\partial \bn_{-}} & \hspace{-0.3cm} + \alpha p^{k-1}_{+}  \\
& & \text{on} \; \gamma \times (0,T),\\
p^{k}_{-}(\cdot,0) & = p_{0} & \text{in} \; \Omega_{-},
\end{array} 
\end{equation}
and
\begin{equation} \label{Append:A3FracSchwarzalg2}
\hspace{-0.2cm}\begin{array}{rll} s_{+}\partial _{t} p^{k}_{+}+\Div (-\bK_{-}\nabla p^{k}_{+})& = q & \text{in} \; \Omega_{+} \times (0,T) , \\
\ds\bK_{+} \frac{\partial p^{k}_{+}}{\partial \bn_{+}} + \alpha p^{k}_{+} +s_{\gamma}\partial _{t} p^{k}_{+} +\Divtau (-\bK_{f, \tau} \delta \nabla_{\tau} p^{k}_{+})  & = \ds\bK_{-} \frac{\partial p^{k-1}_{-}}{\partial \bn_{+}} & \hspace{-0.3cm} + \alpha p^{k-1}_{-}  \\
& & \text{on} \; \gamma \times (0,T),\\
p^{k}_{+}(\cdot,0) & = p_{0} & \text{in} \; \Omega_{+},
\end{array} 
\end{equation}
where $ \gamma =  \{ x=0 \} $ is the fracture. We assume that the permeability is isotropic:
$$ \bK_{pm} = \mathfrak{K}_{\pm} \bI , \; \; \text{and} \; \bK_{f,\tau} = \mathfrak{K}_{f}, $$
 where $ \bI $ is the 2D identity matrix, and that the solution of the problem decays at infinity. As the problem is linear, we only consider $ q = 0 $ and $ p_{0} = 0 $, and analyse the convergence of \eqref{Append:A3FracSchwarzalg1}-\eqref{Append:A3FracSchwarzalg2} to the zero solution.  
We use a Fourier transform in time and in the $ y $ direction
with parameters $ \omega $ and $ \eta $, respectively, to obtain the Fourier functions $ \hat{p}^{k}_{\pm} $ in time $ t $ and $ y $ of $ p^{k}_{\pm} $, as the solutions to the ordinary differential equation in $ x $
\begin{equation*}
-\mathfrak{K} \frac{\partial^{2} \hat{p}}{\partial x^{2}} + \left (s i\omega + \mathfrak{K} \eta^{2} \right ) \hat{p}= 0. 
\end{equation*}
Thus
$$ \hat{p} = A (\eta, \omega) e^{r^{+}x} +B (\eta, \omega) e^{r^{-}x},
$$
where $ r^{\pm} $ are the roots of the characteristic equation
$$ -\mathfrak{K} r^{2} + \left (s i\omega + \mathfrak{K} \eta^{2} \right ) =0,
$$
so
$$ r^{\pm} = \pm \frac{\sqrt{\Delta}}{2 \mathfrak{K}}, \quad \Delta = 4 \mathfrak{K} \left (s i\omega + \mathfrak{K} \eta^{2} \right ).
$$
Here and throughout this article, we use the square root symbol $\sqrt{\hspace{1.2mm}}$ to denote the complex square root with positive real part.
In order to work with at least square integrable functions in time and space, we look
for solutions which do not increase exponentially in $x$. Since $\Re r^+ > 0$ and $\Re r^{-} < 0$,
we obtain
\begin{equation*} \label{Append:A3FracODEsol}
\begin{array}{ll}
\hat{p}^{k}_{-} & = A^{k}(\eta, \omega) e^{r^{+}(s_{-}, \mathfrak{K}_{-}, \eta, \omega) x}, \\
\hat{p}^{k}_{+} &= B^{k} (\eta, \omega) e^{r^{-}(s_{+}, \mathfrak{K}_{+}, \eta, \omega) x}.
\end{array}
\end{equation*}
Substituting these formulas into the transmission conditions on the interface $ \gamma \times (0,T) $ (i.e. the second equations of \eqref{Append:A3FracSchwarzalg1} and \eqref{Append:A3FracSchwarzalg2}), we find
\begin{equation} \label{Append:A3FracTCsFourier}
\begin{array}{rl}
\left( \mathfrak{K}_{-} r^{+}(s_{-}, \mathfrak{K}_{-}, \eta, \omega) + \alpha + s_{\gamma} i \omega + \mathfrak{K}_{f} \delta \eta^{2}\right) & \hat{p}^{k}_{-}(0,\eta, \omega) \vspace{0.2cm}\\ 
& \hspace{-1.1cm} = \left( \mathfrak{K}_{+} r^{-}(s_{+}, \mathfrak{K}_{+}, \eta, \omega) + \alpha \right) \hat{p}^{k-1}_{+} (0,\eta, \omega),\vspace{0.3cm}\\
\left(- \mathfrak{K}_{+} r^{-}(s_{+}, \mathfrak{K}_{+}, \eta, \omega) + \alpha + s_{\gamma} i \omega + \mathfrak{K}_{f} \delta \eta^{2} \right) & \hat{p}^{k}_{+} (0,\eta, \omega)\vspace{0.2cm}\\ 
& \hspace{-1.1cm} =\left( -\mathfrak{K}_{-} r^{+}(s_{-}, \mathfrak{K}_{-}, \eta, \omega) + \alpha \right) \hat{p}^{k-1}_{-}(0,\eta, \omega).
\end{array}
\end{equation}
From \eqref{Append:A3FracTCsFourier} using induction ans using
$ \zeta$ to denote $s_{\gamma} i \omega + \mathfrak{K}_{f} \delta \eta^{2},
$
we obtain
\begin{align*}
\hat{p}^{2k}_{-}(0,\eta, \omega) & = \frac{ \mathfrak{K}_{+} r^{-}(s_{+}, \mathfrak{K}_{+}, \eta, \omega) + \alpha }{\mathfrak{K}_{-} r^{+}(s_{-}, \mathfrak{K}_{-}, \eta, \omega) + \alpha +\zeta } \quad \hat{p}^{2k-1}_{+} (0,\eta, \omega)\\
& = \left (\frac{ \mathfrak{K}_{+} r^{-}(s_{+}, \mathfrak{K}_{+}, \eta, \omega) + \alpha }{\mathfrak{K}_{-} r^{+}(s_{-}, \mathfrak{K}_{-}, \eta, \omega) + \alpha +\zeta }\right ) \;  \left (\frac{-\mathfrak{K}_{-} r^{+}(s_{-}, \mathfrak{K}_{-}, \eta, \omega) + \alpha }{- \mathfrak{K}_{+} r^{-}(s_{+}, \mathfrak{K}_{+}, \eta, \omega) + \alpha +\zeta }\right ) \quad \hat{p}^{2k-2}_{-}(0,\eta, \omega) \\
& =  \rho^{k}_{\text{f}} \hat{p}^{0}_{-}(0,\eta, \omega),
\end{align*}
where
\begin{equation*}
\rho_{\text{f}} = \left (\frac{ \mathfrak{K}_{+} r^{-}(s_{+}, \mathfrak{K}_{+}, \eta, \omega) + \alpha }{\mathfrak{K}_{-} r^{+}(s_{-}, \mathfrak{K}_{-}, \eta, \omega) + \alpha +\zeta }\right ) \; \left ( \frac{-\mathfrak{K}_{-} r^{+}(s_{-}, \mathfrak{K}_{-}, \eta, \omega) + \alpha }{- \mathfrak{K}_{+} r^{-}(s_{+}, \mathfrak{K}_{+}, \eta, \omega) + \alpha +\zeta) }\right ),
\end{equation*}
is the convergence factor of the algorithm~\eqref{Append:A3FracSchwarzalg1}-\eqref{Append:A3FracSchwarzalg2}. 
Similarly, we obtain
$$ \hat{p}^{2k}_{+}(0,\eta, \omega) = \rho_{f}^{k} \hat{p}^{0}_{+}(0,\eta, \omega),
$$
Thus, we can calculate the parameter $ \alpha $ in such a way as to minimize this continuous convergence factor:
\begin{equation}\label{Append:A3FracnumConvFact}
\min_{\alpha >0} \left ( \max_{\lvert \eta \rvert \in \left [\frac{\pi}{L}, \frac{\pi}{h}\right ], \lvert \omega \rvert \in \left [\frac{\pi}{T}, \frac{\pi}{\Delta t}\right ]} \big\vert \rho_{\text{f}}  (s_{+}, \mathfrak{K}_{+}, s_{-}, \mathfrak{K}_{-}, \alpha, \eta, \omega)
\big\vert \right),
\end{equation}
where $ L $ is the length of the fracture, $ h $ is the spatial mesh size, $ T $ is the final time and $ \Delta t $ is the maximum time step of the discretization in time. 
\begin{remark}
One could make use of the two-sided Robin as in~\cite{PhuongSINUM}. In this article, the optimized one-sided Robin parameter works well since in the test case we considered, the two subdomains $ \Omega_{1} $ and $ \Omega_{2} $ (representing the rock matrix) have similar physical properties (though a comparison of the performance of the one-sided and two-sided Robin might be considered). 
\end{remark}
In our applications, the fracture is assumed to have much larger permeability than the surrounding domain, which suggests that the time step inside the fracture should be small compared with that of the surrounding matrix subdomains. As both of the methods derived in Sections~\ref{A3Sub:SchurM1} and \ref{A3subsec:VentcellDiff} are global in time, i.e. the subdomain problem is solved over the whole time interval before the information is exchanged on the space-time interface, we can use different time steps in the fracture and in the rock matrix.  In the next section, we consider
the semi-discrete problem in time with nonconforming time grids. 
%
%
%
%
\section{Nonconforming discretization in time}
\label{A4Subsec:TimeNonc}
Let $ \iT_{1}, \iT_{2} $ and $ \iT_{\gamma} $
be three different partitions of the time interval $ (0,T) $ into subintervals $ J^{i}_{m}=(t_{m-1}^{i}, t_{m}^{i}] $ for $ m=1, \hdots, M_{i} $, and $ i=1,2, \gamma $,  (see Figure~\ref{A3Fig:TimeFrac}). For simplicity, we consider uniform partitions only, and denote by $ \Delta t_{i}, \; i=1,2,\gamma $ the corresponding time steps. Assume that $ \Delta t_{\gamma} \ll \Delta t_{i}, \; i=1,2. $
We use the lowest order discontinuous Galerkin method \cite{BlayoHJ, OSWRDG,  thomee1997galerkin},
which is a modified backward Euler method. The same idea can be generalized to higher order methods.
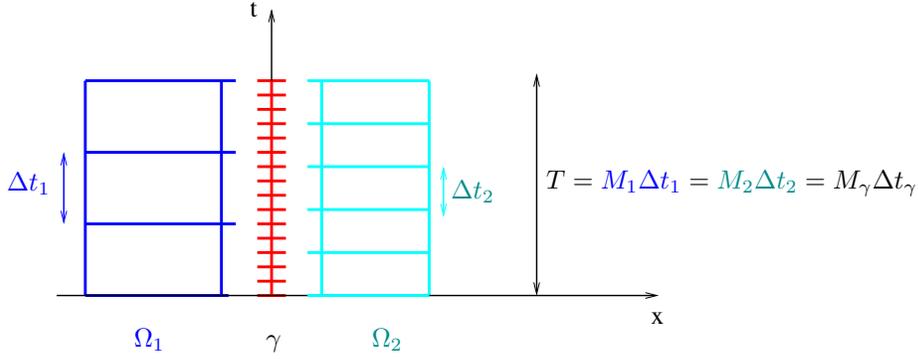
\begin{figure}[htbp]
\centering
\begin{minipage}[c]{0.8\linewidth}
\setlength{\unitlength}{1pt} 
\begin{picture}(140,120)(0,0)
\thicklines
\put(0,3){\includegraphics[scale=0.6]{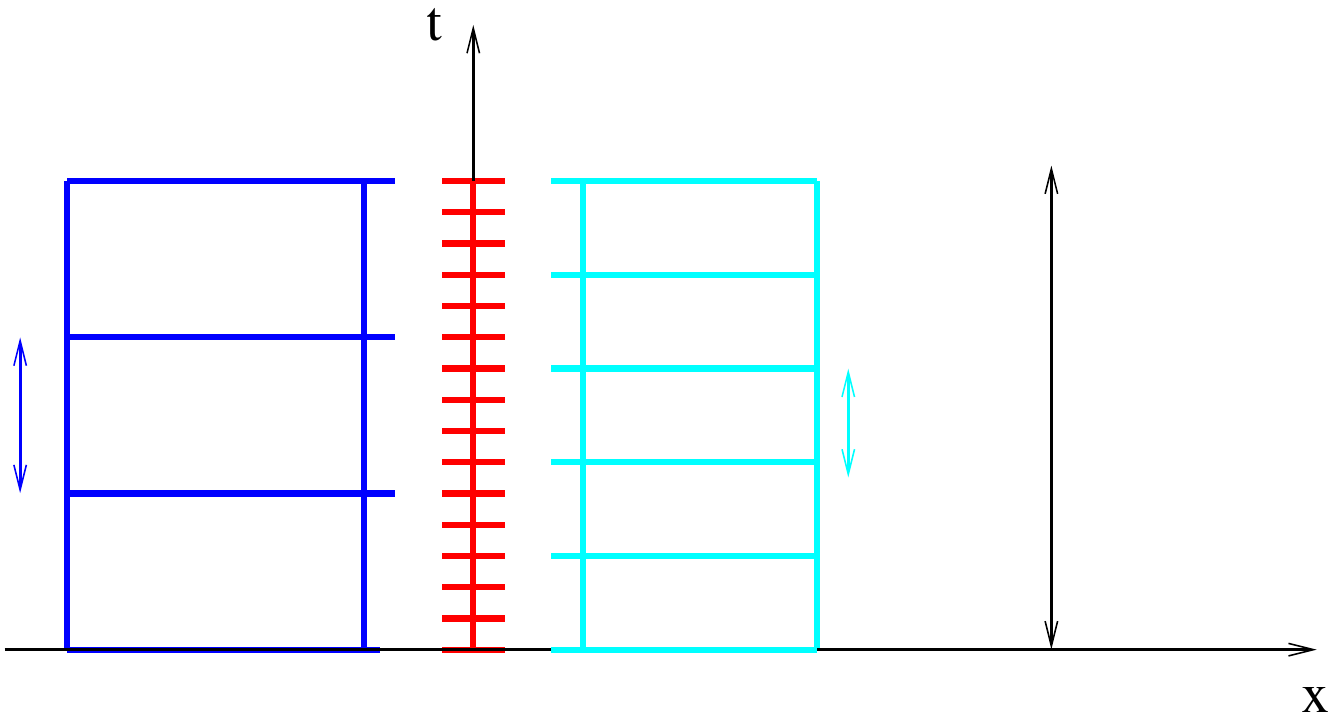} \\}
\put(30,-5){$ \textcolor{blue}{\Omega_{1}} $}
\put(120,-5){$\textcolor{darkcyan}{ \Omega_{2}} $}
\put(80,-5){$ \gamma $}
\put(186,55){$ T = \textcolor{blue}{M_{1} \Delta t_{1}} = \textcolor{darkcyan}{M_{2} \Delta t_{2}} = M_{\gamma} \Delta t_{\gamma}$}
\put(-18,54){$\textcolor{blue}{\Delta t_{1}} $}
\put(150,51){$\textcolor{darkcyan}{\Delta t_{2}} $}
\end{picture}
\end{minipage} \vspace{0.2cm}
\caption{Nonconforming time grids in the rock matrix and in the fracture.}
\label{A3Fig:TimeFrac} 
\end{figure}

We denote by $ P_{0}(\mathcal{T}_{i}, L^{2}(\gamma)) $ the space of functions piecewise constant in time on grid
$ \mathcal{T}_{i} $ with values in $ L^{2}(\gamma) $:
\begin{equation*} 
P_{0}(\mathcal{T}_{i}, L^{2}(\gamma)) = \left \{ \psi: (0,T) \rightarrow L^{2}(\gamma),
\psi \; \; \text{is constant on} \;  J, \ \forall J \in \mathcal{T}_{i} \right \}. \vspace{-0.1cm}
\end{equation*}
In order to exchange data on the space-time interface between different time grids, we use,
for $i,j$ in $\{1,2,\gamma\}$, the
$ L^{2} $ projection $ \Pi_{j i} $ from  $ P_{0} (\mathcal{T}_{i}, L^{2}(\gamma)) $ to $ P_{0}(\mathcal{T}_{j},L^{2}(\gamma)) $:
for $ \psi \in P_{0} (\mathcal{T}_{i},L^{2}(\gamma) )$,
$ \Pi_{ji} \psi \hspace{-2pt} \mid_{J^{j}_{m}} $ is the average value of $ \psi $ on $ J^{j}_{m} $,
for $ m=1, \dots, M_{j} $. 

\subsection*{For the GTP-Schur method} 
The unknown $ \lambda $ on the interface represents the fracture pressure, thus $ \lambda $ is piecewise constant in time
on grid $ \iT_{\gamma} $. In order to obtain Dirichlet boundary data for solving subdomain problem~\eqref{A3subprobM1}, we project $ \lambda $ into $P_{0}(\mathcal{T}_{i},L^{2}(\gamma))$, for $ i=1,2: $
 $$ p_{i} = \Pi_{i\gamma} (\lambda),  \quad \text{on} \; \gamma, \; i=1,2.$$ 
The semi-discrete counterpart of the interface problem~\eqref{A3finterfaceM1-short} is obtained by weakly enforcing the fracture problem over each time sub-interval of $ \iT_{\gamma} $ as follows
\begin{eqnarray} \label{A3NoncTime-M1}
    \ds
s_{\gamma} \left (\lambda^{m+1} - \lambda^{m}\right ) + \int_{t^{m}_{\gamma}}^{t^{m+1}_{\gamma}} \Divtau \bu_{\gamma}^{m+1} 
\ds=\int_{t^{m}_{\gamma}}^{t^{m+1}_{\gamma}} \bigg (\sum_{i=1}^{2} \Pi_{\gamma i}  \left (\iS^{\text{DtN}}_{i}(\Pi_{i\gamma} (\lambda), q_{i}, p_{0,i}) \right ) \bigg), \nonumber\\
\begin{array}{rll}
\bu_{\gamma}^{m+1}& = -\bK_{\gamma} \delta \nabla_{\tau} \lambda^{m+1} & \hspace{1cm} \text{in} \; \gamma, \\
\lambda^{m+1} & = 0 & \hspace{1cm} \text{on} \; \partial \gamma,\\
\lambda^{0} & = p_{0,\gamma} & \hspace{1cm} \text{in} \; \gamma,
\end{array}\hspace{1cm}
\end{eqnarray}
where $\lambda^{m}=\lambda_{J^{\gamma}_m}$, for $ m=0, \hdots, M_{\gamma}-1 $. 

For a function piecewise constant in time $ \varphi $ on the fine grid $ \iT_{\gamma} $, the semi-discrete Neumann-Neumann preconditioner (still denoted by $ \bP_{NN}^{-1} $) is defined by:
\begin{equation} \label{A3NoncTime-NNPrecond}
\bP_{NN}^{-1} \varphi := \sum_{i=1}^{2} \sigma_{i} \Pi_{\gamma i} \left ( \left (\check{\iS}_{i}^{DtN}\right )^{-1} \left
(\Pi_{i \gamma} \left (\varphi \right ) \right) \right ),
\end{equation}
where we have solved the subdomain problem with Neumann-Neumann data projected from $ \iT_{\gamma} $ onto $ \iT_{i} $, $ i=1,2, $ then extracted the pressure trace on the interface and projected backward from $ \iT_{i} $ onto $ \iT_{\gamma} $. Thus the interface problem is defined on the fracture time grid.
\begin{remark} \label{A3rmrkNNprecond}
  In the nonconforming semi-discrete (in time) case,
we see  from~\eqref{A3NoncTime-NNPrecond} that $\bP_{NN}^{-1} $ is not strictly a preconditioner 
and may affect the accuracy of the scheme due the projection operators used to define it.
This is indeed what we observed in the
numerical experiments in Section~\ref{A3Sec:Num}.
\end{remark}
\subsection*{For the GTO-Schwarz method}
In the GTO-Schwarz method, there are two interface unknowns representing the linear combination of the fracture pressure and some terms from the fracture problem.
Thus we let $\theta_{\gamma,j} \in P_{0} (\iT_{\gamma}, L^{2}(\Gamma)) $, for $ j=1,2 $.  In order to obtain Ventcell boundary data for solving the subdomain problem~\eqref{A3subprobM2-short}, we project $ \theta_{\gamma,j} $ onto the $ \iT_{i} $, for $ i=1,2; \ j=3-i: $
 $$ -\bu_{i} \cdot \bn_{i} + \alpha p_{i} +s_{\gamma}\partial _{t} p_{i} +\Divtau \bu_{\gamma}  =
\Pi_{i\gamma} (\theta_{\gamma,j}), \quad \text{on} \; \gamma, \; i=1,2.$$ 
\begin{remark}
This setting is different from the case of usual domain decomposition (without fractures) analyzed in~\cite{PhuongThesis,PhuongSINUM}, where the two interface unknowns represent the Robin data in each subdomain and thus are chosen to be constant on the associated subdomain's time grid, i.e. $ \theta_{\gamma,i} \in P_{0} (\iT_{i}, L^{2}(\Gamma)) $, for $ i=1,2 $. 
\end{remark}

The semidiscrete-in-time counterpart
of~\eqref{A3finterfaceM3} is weakly enforced over each time sub-interval of the fracture time grid as follows:
for all $  m=0, \hdots, M_{\gamma}-1, $
\begin{equation} \label{A3NoncTime-M2}
\begin{array}{rl} 
\ds\int_{t^{m}_{\gamma}}^{t^{m+1}_{\gamma}} \theta_{\gamma,1} &= \ds\int_{t^{m}_{\gamma}}^{t^{m+1}_{\gamma}}  \Pi_{\gamma  2} \left ( \iS^{\text{VtR}}_{2}(\Pi_{2\gamma}(\theta_{\gamma,2}), q_{2}, p_{0,1},  p_{0,\gamma})\right ), \vspace{0.12cm}\\
\ds\int_{t^{m}_{\gamma}}^{t^{m+1}_{\gamma}}\theta_{\gamma,2} & =\ds \int_{t^{m}_{\gamma}}^{t^{m+1}_{\gamma}}  \Pi_{\gamma 1} \left ( \iS^{\text{VtR}}_{1}(\Pi_{1\gamma}(\theta_{\gamma,1}), q_{1}, p_{0,1}, p_{0,\gamma})\right ),
\end{array} \qquad  \text{on} \; \gamma, 
\end{equation}
\begin{remark} \label{A3rmrkOSWRaccuracy}
We point out that with the GTO-Schwarz method as with the GTP-Schur method preconditioned by a Neumann-Neumann preconditioner (cf. Remark~\ref{A3rmrkNNprecond}), we can not hope to gain in accuracy in the fracture by using a finer grid
in the fracture only
since the fracture problem is actually solved on the coarser time grids of the two subdomains. We will see this in the numerical experiments.
\end{remark}
\section{Numerical results}
\label{A3Sec:Num}
In all of the numerical experiments, for the spatial discretization we use mixed finite elements with the
lowest order Raviart-Thomas spaces on rectangles~\cite{brezzi1991mixed,RobertsThomas}.
\begin{remark} \label{A3rmrk:CostVent}
The subdomain problem of the GTO-Schwarz method corresponding to Ventcell boundary conditions is somewhat more complicated than that of GTP-Schur method (problem~\eqref{A3subprobM1}). Consequently, for solving problem~\eqref{A3subprobM2}, one needs to introduce Lagrange multipliers (see e.g.~\cite{brezzi1991mixed,RobertsThomas}) on the interface to handle the Ventcell conditions (representing the fracture problem). 
\end{remark}

We carry out some preliminary experiments to investigate the numerical performance of the two methods proposed above. We consider the test case pictured in Figure~\ref{A3Fig:testGeo} where the domain is a rectangle of dimension $ 2 \times 1 $ and is divided into two equally sized subdomains by a fracture of width $ \delta = 0.001 $ parallel to the $ y $ axis. The permeability tensors in the subdomains and in the fracture are isotropic: $ \bK = \mathfrak{K}_{i} \bI, \; i=1,2,f, $ and $ \mathfrak{K}_{i} $ is assumed to be constant. Here we choose $ \mathfrak{K}_{1} = \mathfrak{K}_{2} = 1 $ and $ \mathfrak{K}_{f} = 10^{3} $ (so that $ \mathfrak{K}_{f} \delta =1 $). 
A pressure drop of $ 1 $ from the bottom to the top of the fracture is imposed.
On the external boundaries of the subdomains a no flow 
boundary condition is imposed except on the lower fifth (length $ 0.2 $) of both lateral boundaries where a Dirichlet condition is imposed: $ p=1 $ on the right and $ p=0 $ on the left. See Figure~\ref{A3Fig:testGeo}.

We consider a uniform rectangular mesh with size $ h=1/100 $. In time,  we fix $ T = 0.5 $ and use uniform time partitions in the subdomains with time step $ \Delta t_{i}, i=1,2, $ and in the fracture with varying time step $ \Delta t_{\gamma} $. We first consider the case with the same time step throughout the domain, $\Delta t_{1} = \Delta t_{2}  = \Delta t_{\gamma} = \Delta t= T/300 $. 

\begin{figure}[htbp] 
\vspace{-0.2cm}\centering
\begin{minipage}[c]{0.4 \linewidth}
\includegraphics[scale=0.7]{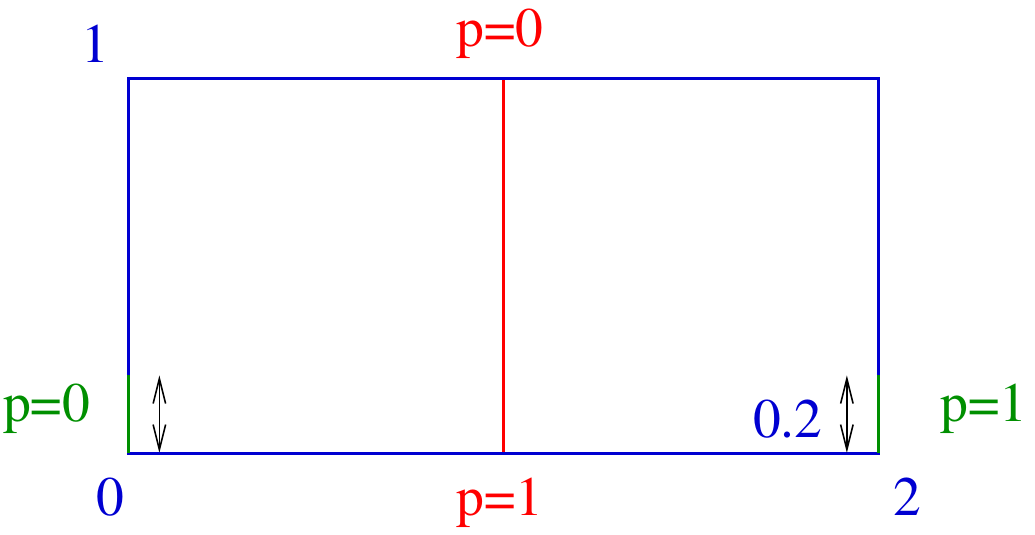} 
\end{minipage} \vspace{-0.3cm}
\caption{Geometry of the test case where the fracture is considered as an interface.} 	
\label{A3Fig:testGeo} \vspace{-0.3cm}
\end{figure}
%
%
In Figure~\ref{A3Fig:CompressSolution}, snapshots at different times of the pressure field and of the flow field (on a coarse grid for visualization)  are shown. The length of each arrow is proportional to 
\begin{figure}[htbp]
\centering
\begin{minipage}[c]{0.45 \linewidth}
\includegraphics[scale=0.25]{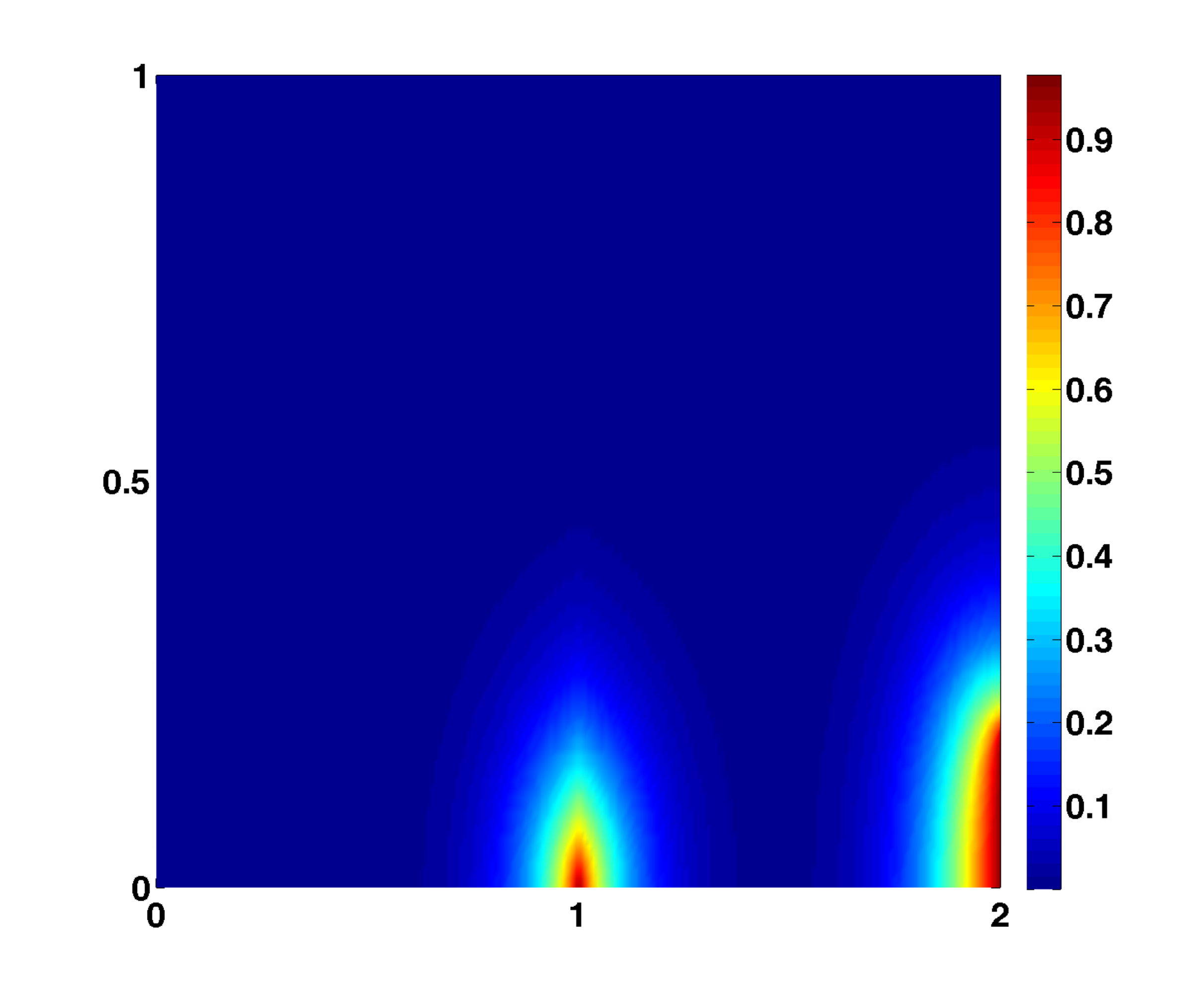}
\end{minipage} \hspace{5pt}
\begin{minipage}[c]{0.45 \linewidth}
\includegraphics[scale=0.25]{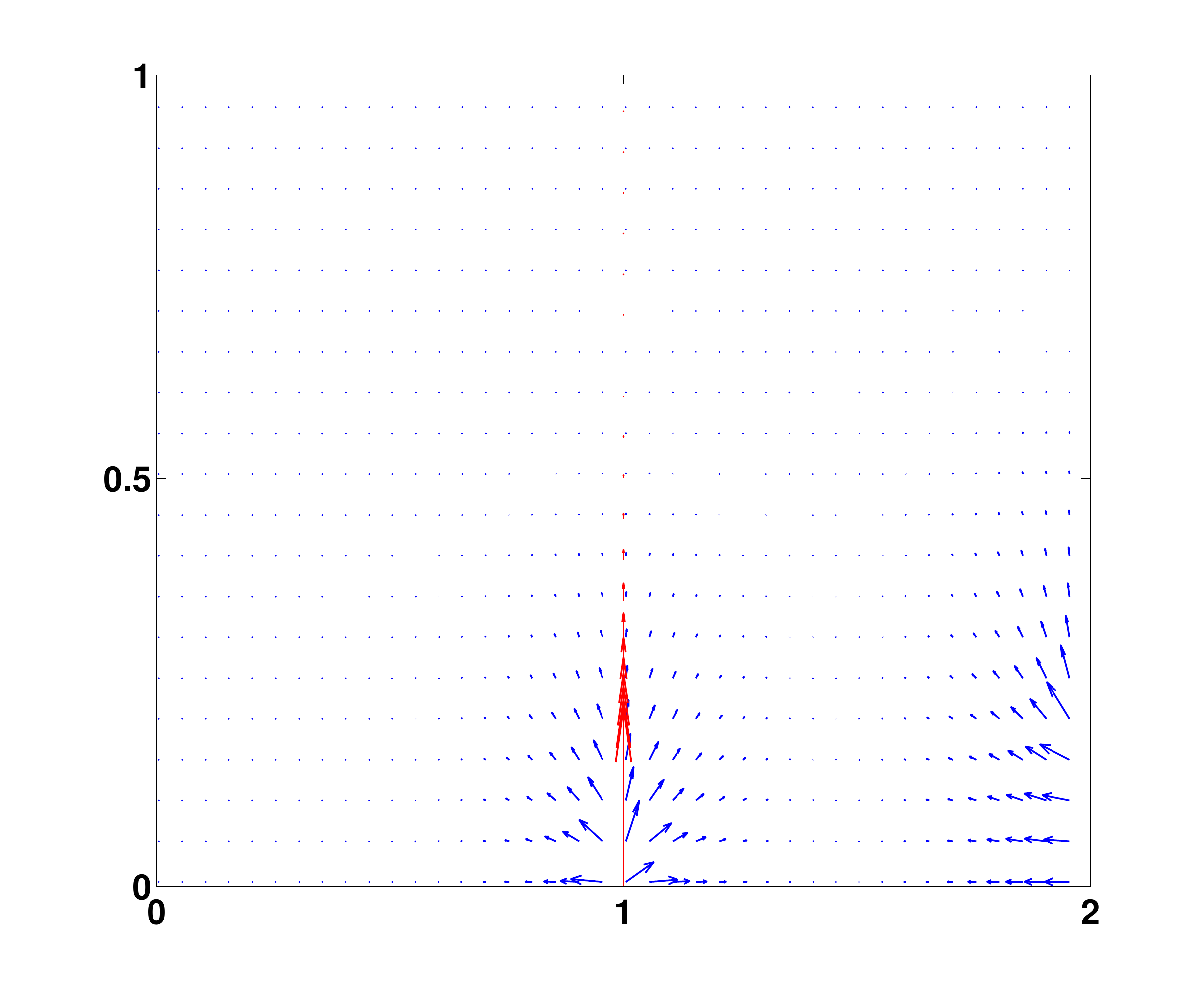}
\end{minipage} \\
\begin{minipage}[c]{0.45 \linewidth}
\includegraphics[scale=0.25]{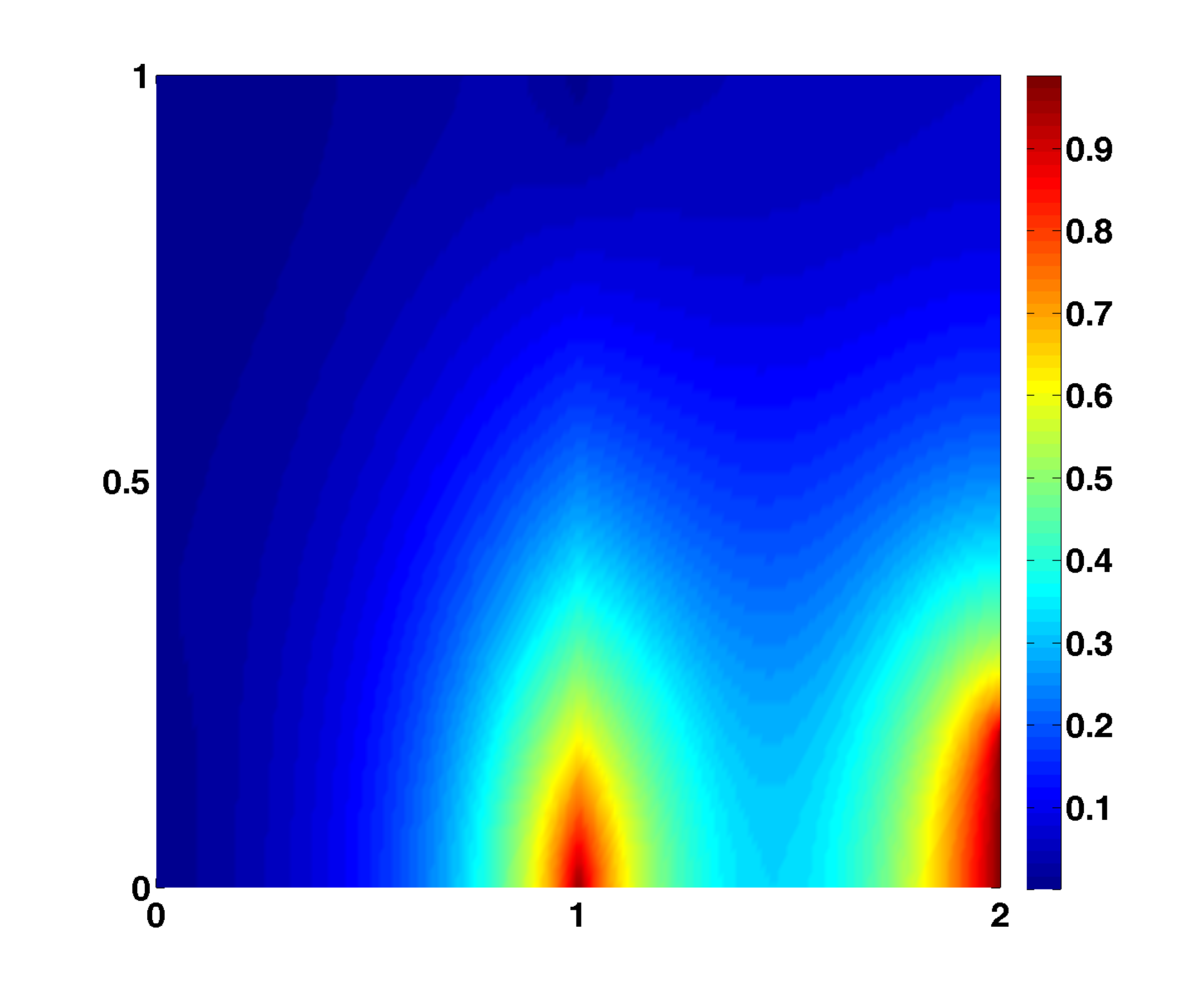}
\end{minipage} \hspace{5pt}
\begin{minipage}[c]{0.45 \linewidth}
\includegraphics[scale=0.25]{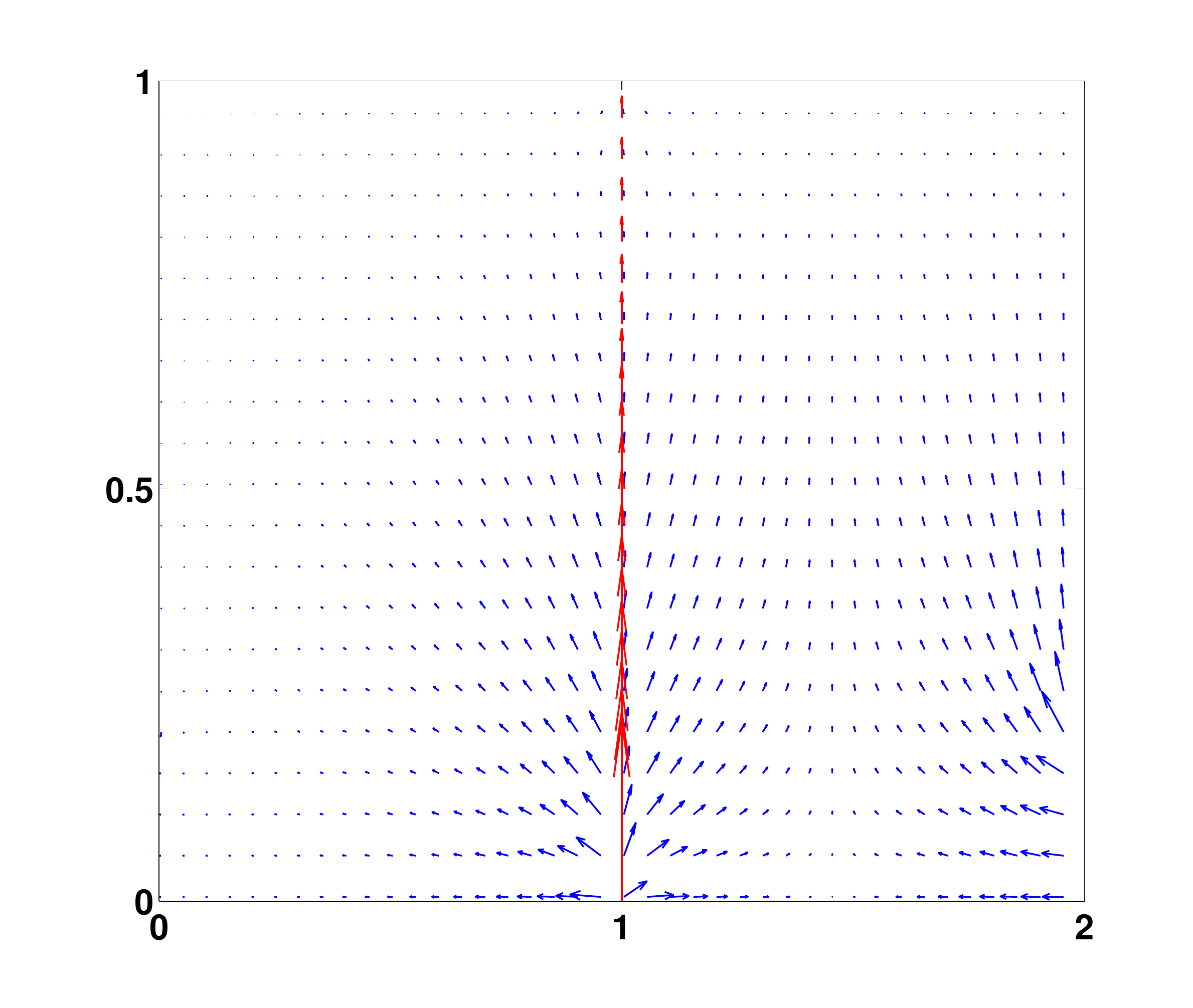}
\end{minipage} \\
\begin{minipage}[c]{0.45 \linewidth}
\includegraphics[scale=0.25]{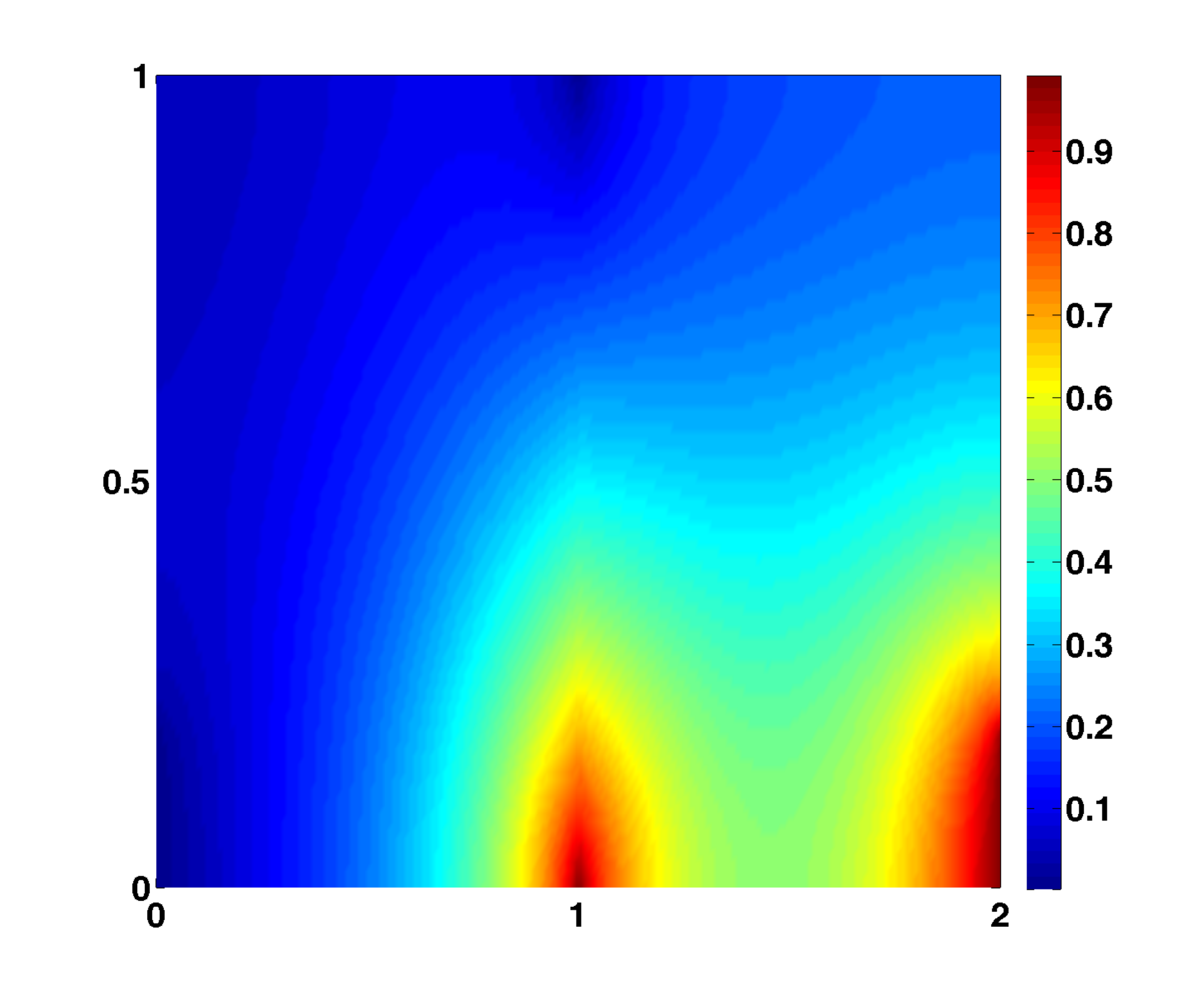}
\end{minipage} \hspace{5pt}
\begin{minipage}[c]{0.45 \linewidth}
\includegraphics[scale=0.25]{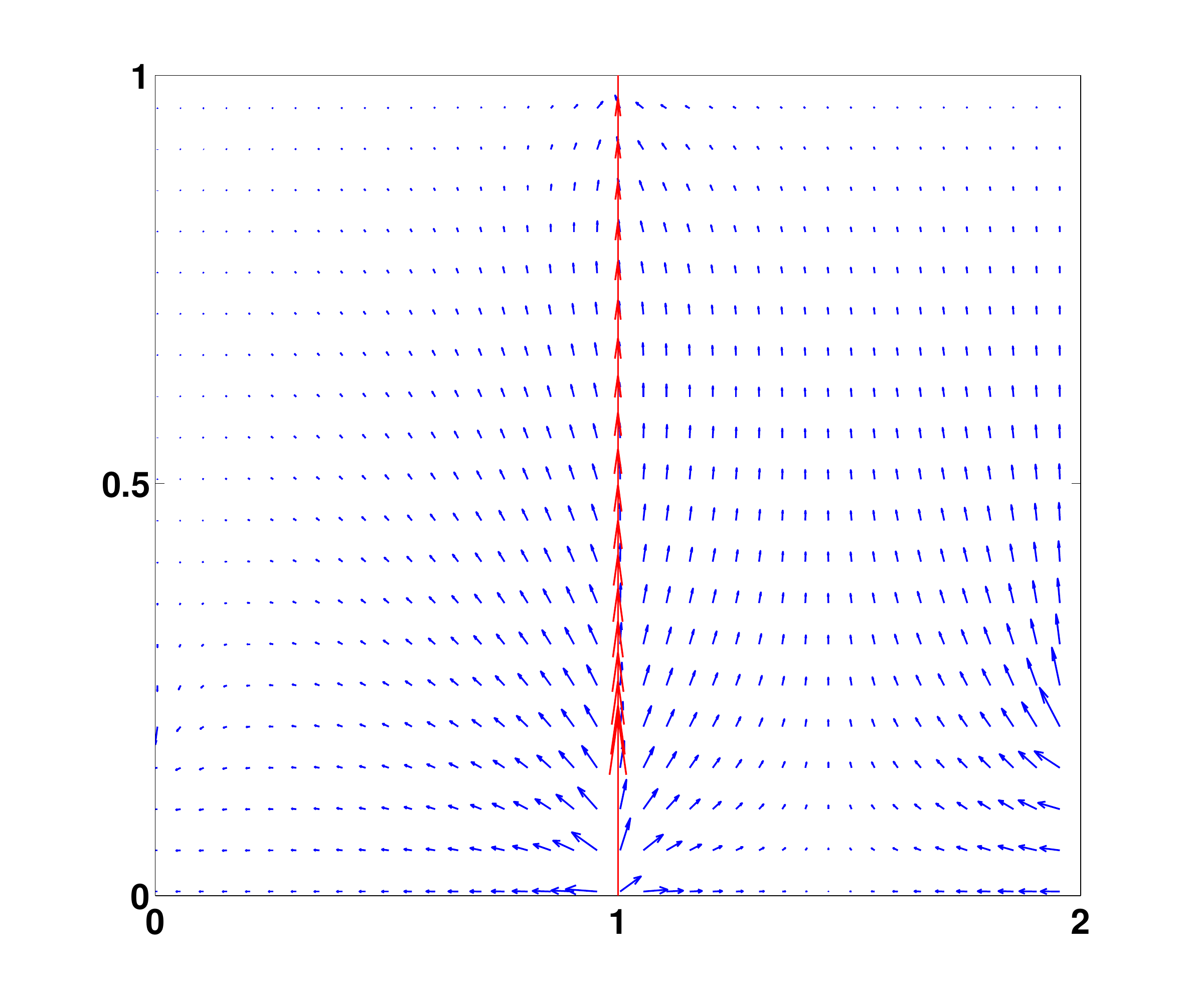}
\end{minipage} \\
\begin{minipage}[c]{0.45 \linewidth}
\includegraphics[scale=0.25]{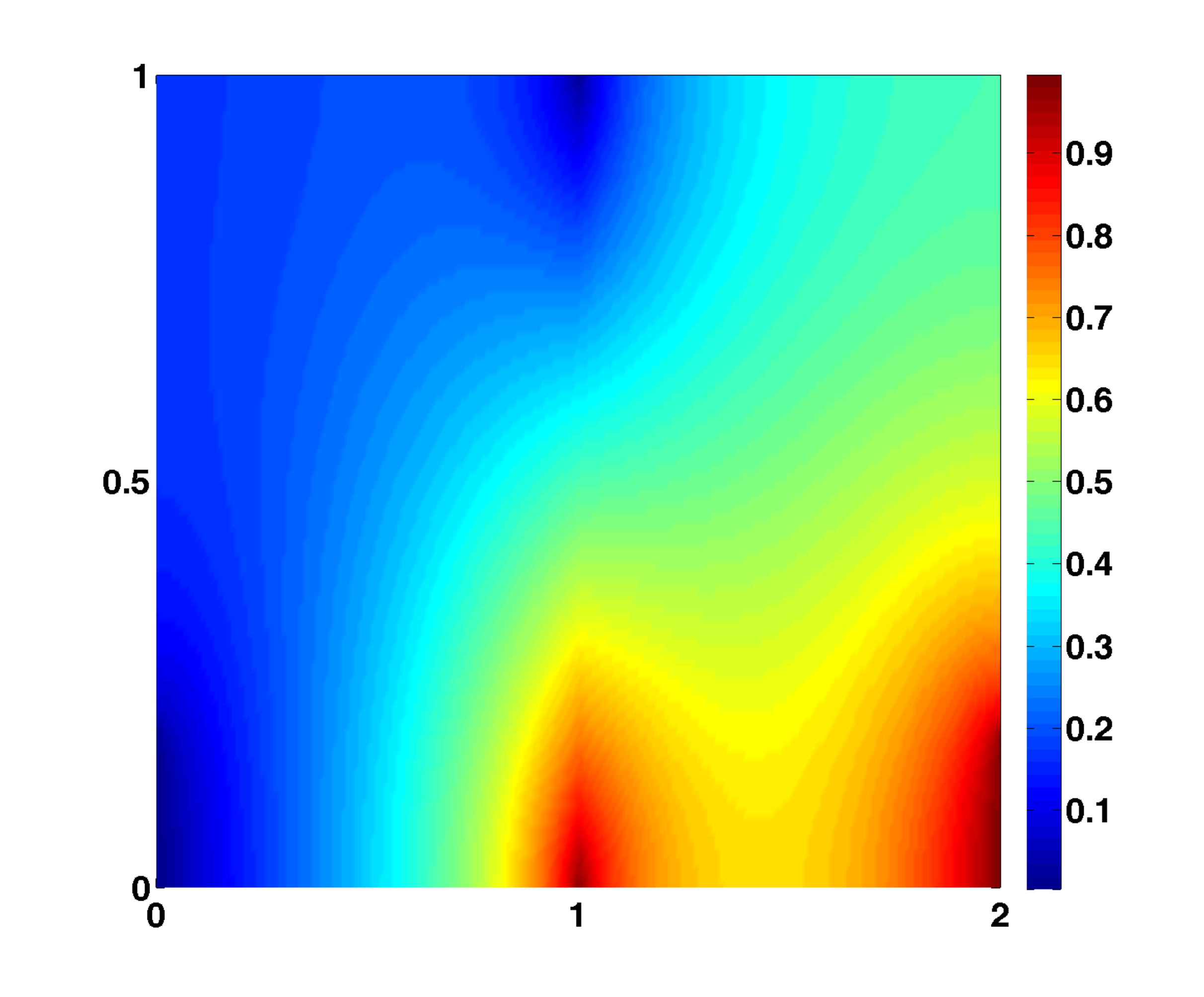}
\end{minipage} \hspace{5pt}
\begin{minipage}[c]{0.45 \linewidth}
\includegraphics[scale=0.25]{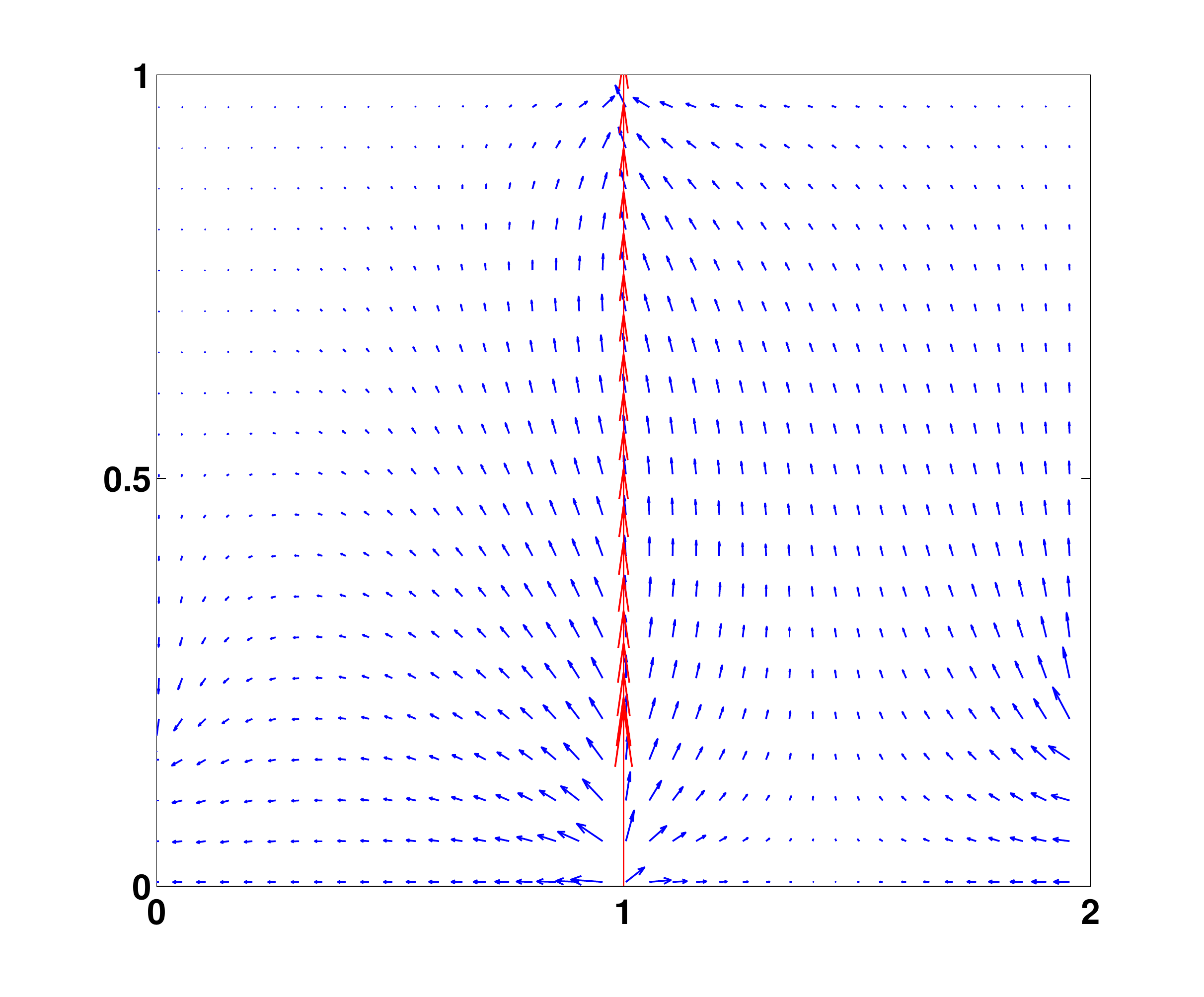}
\end{minipage}
\vspace{-0.2cm}
\caption{Snapshots of the pressure field (left) and flow field (right) at $ t=T/300 $, $ t= T/4 $, $ t=T/2 $ and $ t=T $ respectively (from top to bottom).} 	
\label{A3Fig:CompressSolution} \vspace{-0.2cm}
\end{figure}	
the magnitude of the velocity it represents and the red arrows represent the flow in the fracture. We see that the flow field is a combination of flow in the fracture and flow going from right to left in the rest of the porous medium and there is interaction between them as some fluid flows out of the fracture (near the bottom) and some flows into it
(near the top at later times). Since $ \mathfrak{K}_{f} \gg \mathfrak{K}_{i}, i=1,2 $, the velocity is much larger in the fracture than in the surrounding medium.

%
%
Next, in order to analyze the convergence behavior of both methods, we consider the problem with homogeneous Dirichlet boundary conditions (i.e., the solution converges to zero). We start with a random initial guess on the space-time interface-fracture and use GMRES as an iterative solver and compute the error in the $ L^{2} (0,T; L^{2}(\Omega)) $-norm for the pressure $ p $ and for the velocity~$ \bu $. We stop the iteration when the relative error is less than $ 10^{-6} $. We consider four algorithms: the GTP-Schur method with no preconditioner, the GTP-Schur method with the local preconditioner, the GTP-Schur method with the Neumann-Neumann preconditioner and the GTO-Schwarz method with the optimized Robin parameter. We compare the convergence behavior of these four algorithms in terms of the number of iterations.
Note however that for the GTP-Schur method with the Neumann-Neumann preconditioner the cost per iteration is roughly
twice as large as that of the other methods.

In Figure~\ref{A3Fig:CompresConv}, the error curves versus the number of iterations are shown: the error in $ p $ (on the left) and in $ \bu $ (on the right). We see that the GTP-Schur method with no preconditioner (the blue curves) converges extremely slowly (after $ 500 $ iterations, the error, both in $ p $ and in $ \bu $, is about $ 10^{-1} $). The performance of the GTP-Schur method with the local preconditioner (the green curves) is better but still quite slow: it requires about $ 350 $ iterations to reach an error reduction of $ 10^{-6} $. The Neumann-Neumann preconditioner (the cyan curves) further improves the convergence rate about $ 150 $ iterations are needed to obtain a similar error reduction. The GTO-Schwarz method needs only $ 6 $ iterations to reduce the error to $ 10^{-6} $ and thus the convergence of the GTO-Schwarz method is much faster than that of the other algorithms (at least by a factor of $ 25 $). This is due to the use of the optimized parameter $ \alpha $. In Figure~\ref{A3Fig:CompresRobin}, we show the error in $ \bu $ (in logarithmic scale) after $ 10 $ Jacobi iterations for various values of $ \alpha $. We see that the optimized Robin parameter (the red star) is located close to those giving the smallest error after the same number of iterations. Also we observe that the convergence can be significantly slower if $ \alpha $ is not chosen well. 
\begin{figure}[htbp]
\centering
\begin{minipage}{0.45 \linewidth}
\includegraphics[scale=0.26]{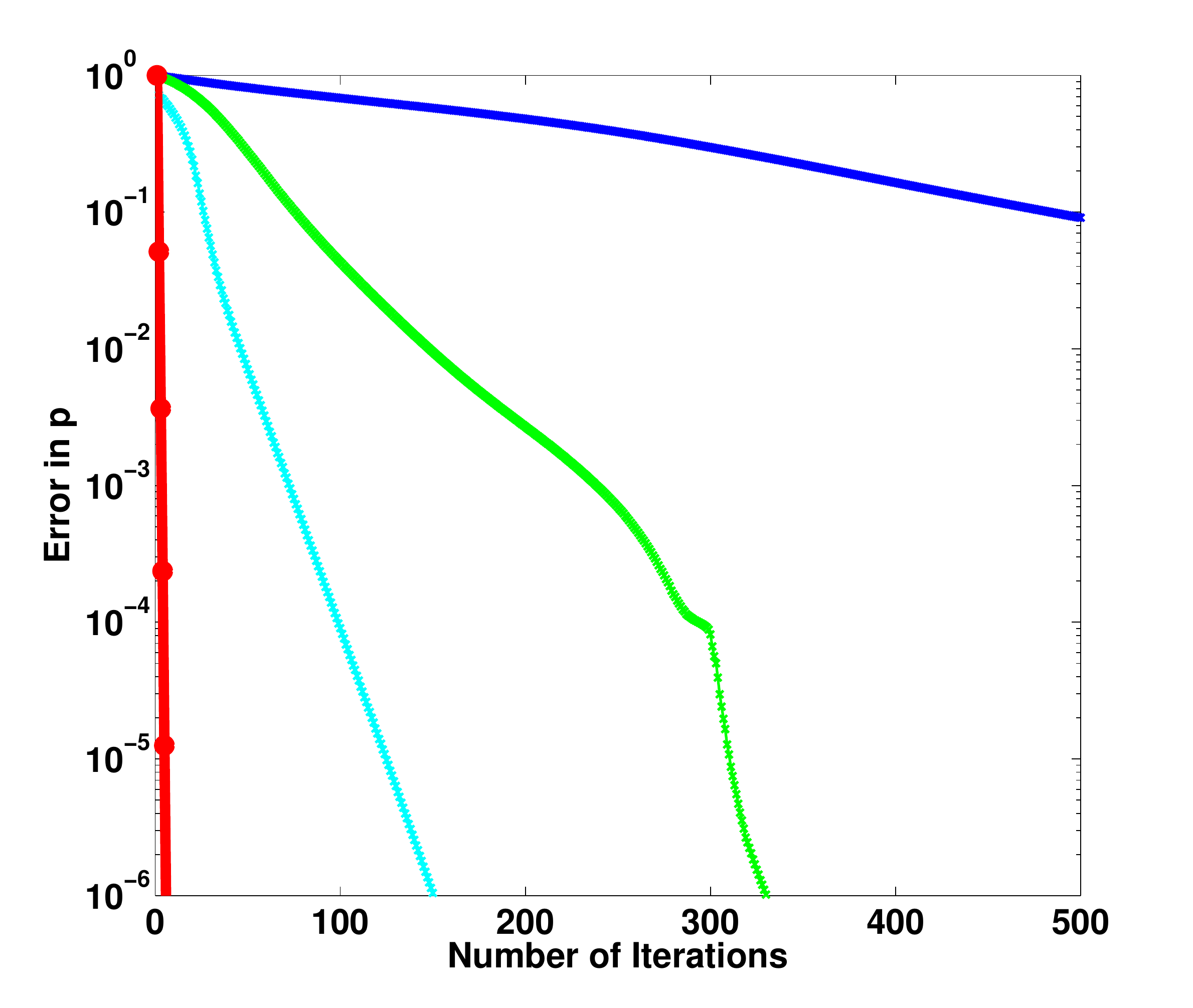}
\end{minipage} \hspace{10pt}
\begin{minipage}{0.45 \linewidth}
\includegraphics[scale=0.26]{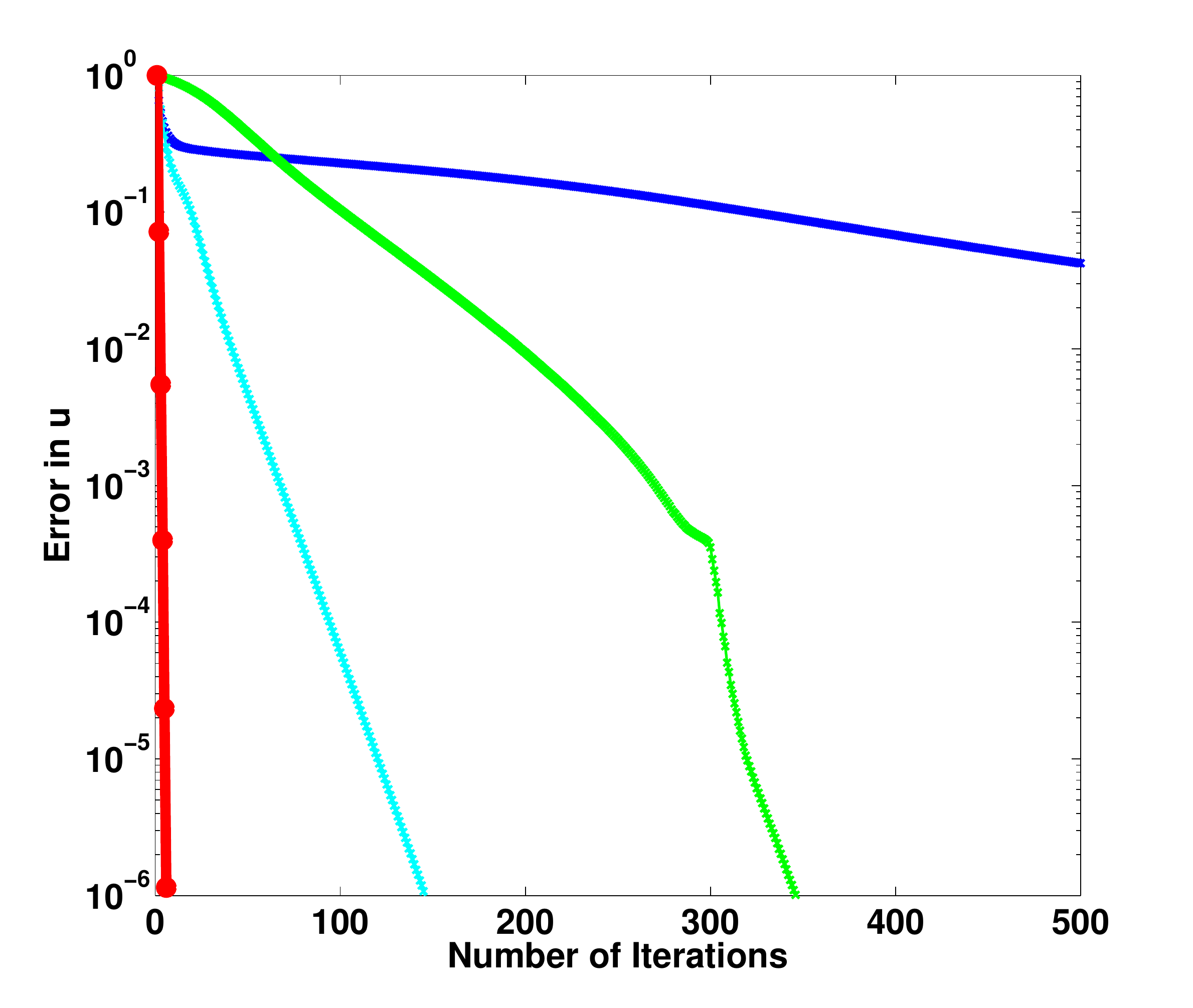} 
\end{minipage} 
\caption{Convergence curves for the compressible flow: errors in $ p $ (on the left) and in $ \bu $ (on the right) -
  GTP-Schur method with no preconditioner (blue), with local preconditioner (green) and
  with Neumann-Neumann preconditioner (cyan) and GTO-Schwarz method (red).} 	
\label{A3Fig:CompresConv} 
\end{figure} 
%
%
%
\begin{figure}[htbp]
\centering
\begin{minipage}{0.45 \linewidth}
\includegraphics[scale=0.27]{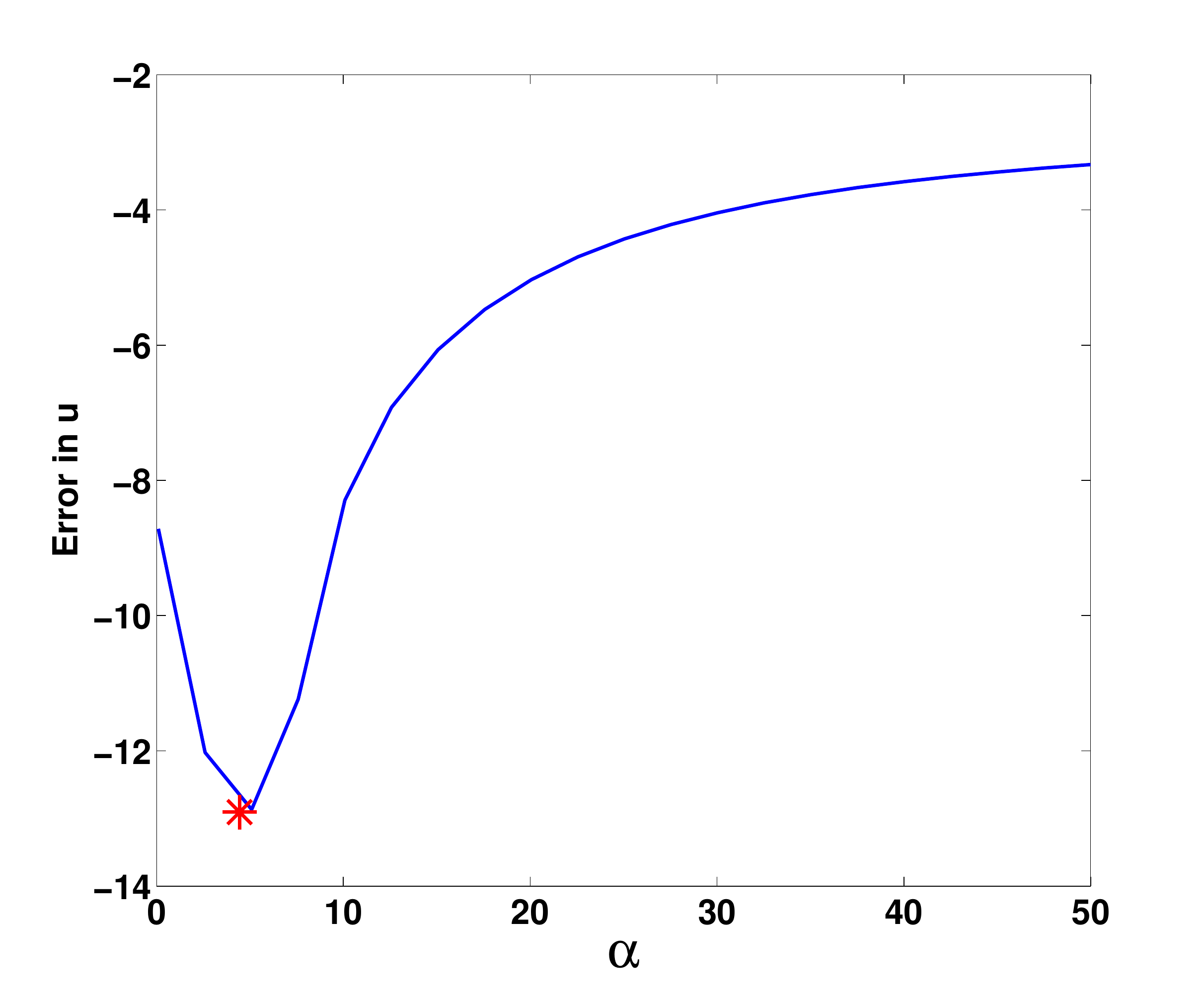}
\end{minipage} 
\caption{$ L^{2} $ velocity error (in logarithmic scale) after $ 10 $ Jacobi iterations for various values
  of the Robin parameter. The red star shows the optimized parameters computed
  by numerically minimizing the continuous convergence factor.} 	
\label{A3Fig:CompresRobin} 
\end{figure}

%
%

Next, we study the behavior of three of the algorithms when nonconforming time grids are used. For this we again use
the nonhomogeneous boundary conditions depicted in Figure~\ref{A3Fig:testGeo}. In all cases, we consider equal time steps for the subdomains as they have the same permeability: $ \Delta t_{1} = \Delta t_{2} = \Delta t_{m} $. We examine three time grids as follows:
\begin{itemize} \itemsep0pt
	\item Time grid 1 (conforming coarse): $ \Delta t_{m} = \Delta t_{f} = T/100 $.
	\item Time grid 2 (nonconforming):  $ \Delta t_{m} = T/100 $ and
              $ \Delta t_{f} =T/500 $.
	\item Time grid 3 (conforming fine): $ \Delta t_{m} = \Delta t_{f} = T/500 $.
\end{itemize}
We start with a zero initial guess on the space-time interface and stop the GMRES iterations when the relative residual is less than $ 10^{-6} $. In Figure~\ref{A3Fig:CompresTimeRelres} we show the relative residual versus the number of iterations for three schemes: the GTP-Schur method with the local preconditioner, the GTP-Schur method with the Neumann-Neumann preconditioner and the GTO-Schwarz method with an optimized Robin parameter. We see that the GTO-Schwarz method still performs better than the GTP-Schur method, and the GTP-Schur method with the Neumann-Neumann preconditioner still converges faster than with the local preconditioner. The convergence rate of both the GTP-Schur method with the Neumann-Neumann preconditioner and the GTO-Schwarz method are almost independent of the time grid (the number of iterations does not change with the time grid) while that of the local preconditioner significantly depends on the temporal grid in the
fracture. We also notice that the behavior of all three methods in the cases of nonconforming and conforming fine grids are very similar. 
\begin{figure}[htbp]
\hspace{-0.3cm}\begin{minipage}[c]{0.30 \linewidth}
\begin{center}
\includegraphics[scale=0.22]{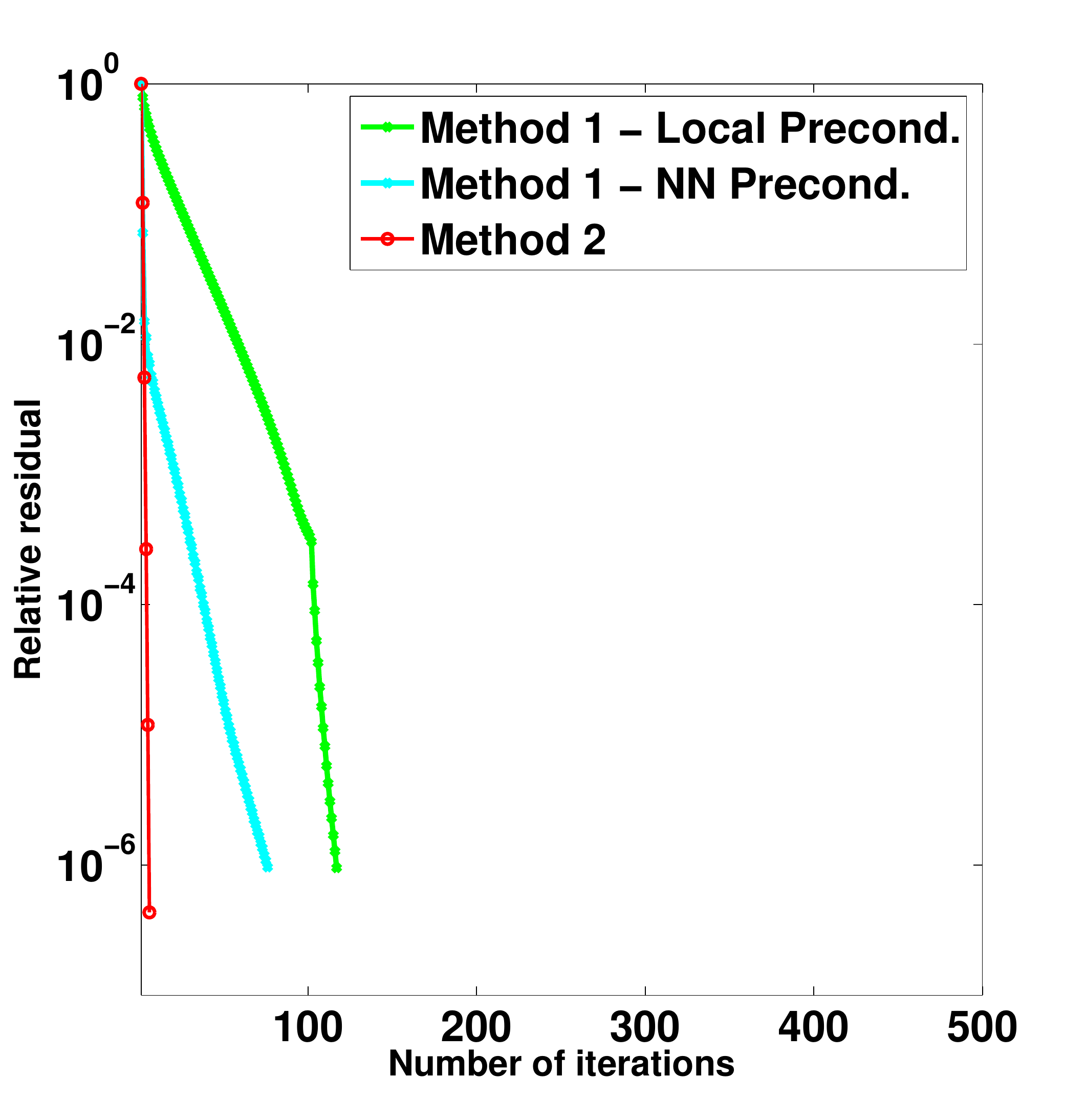}
\end{center}
\end{minipage} \hspace{7pt}
\begin{minipage}[c]{0.30 \linewidth}
\begin{center}
\includegraphics[scale=0.22]{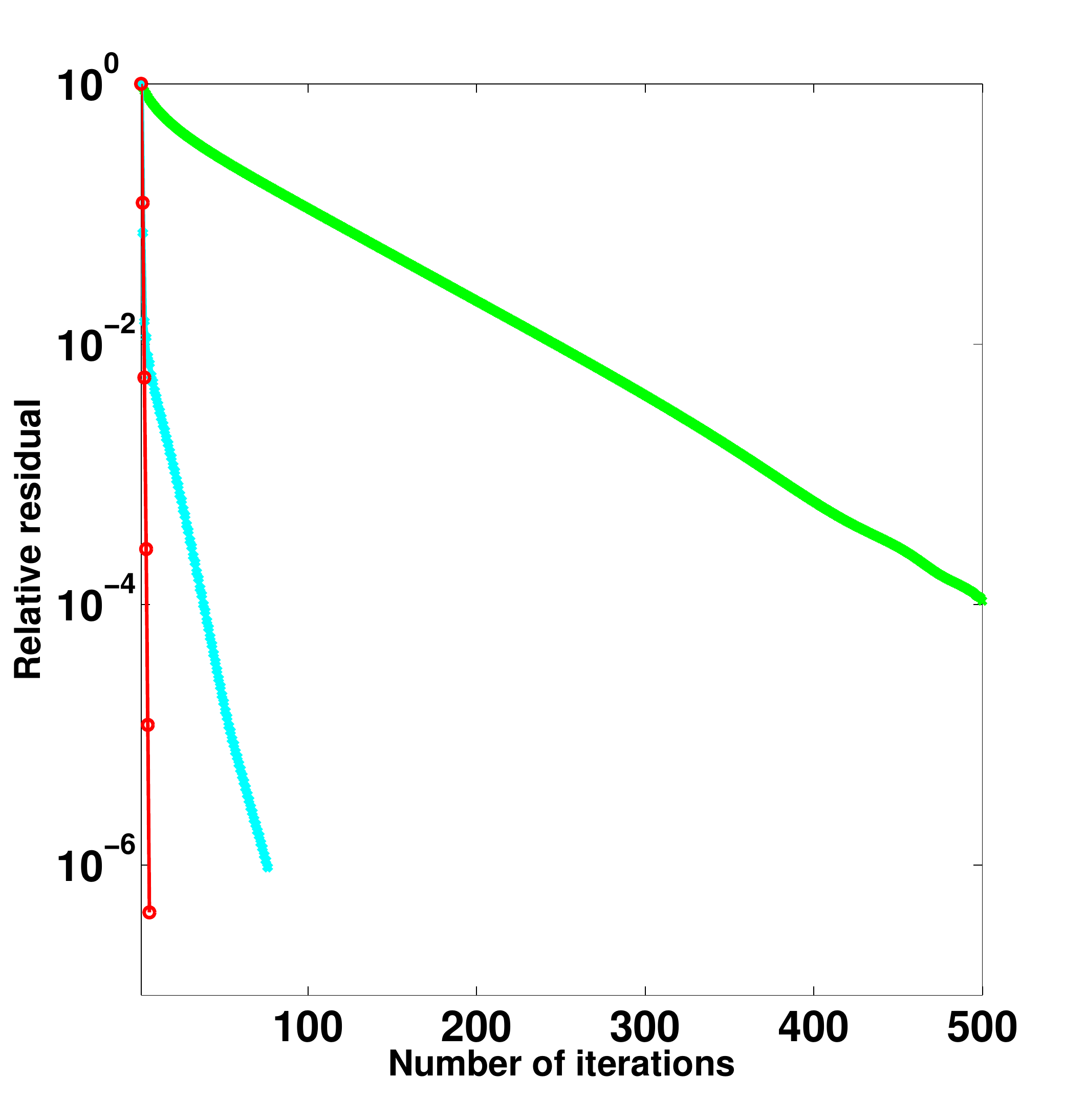}
\end{center}
\end{minipage} \hspace{7pt}
\begin{minipage}[c]{0.30 \linewidth}
\begin{center}
\includegraphics[scale=0.22]{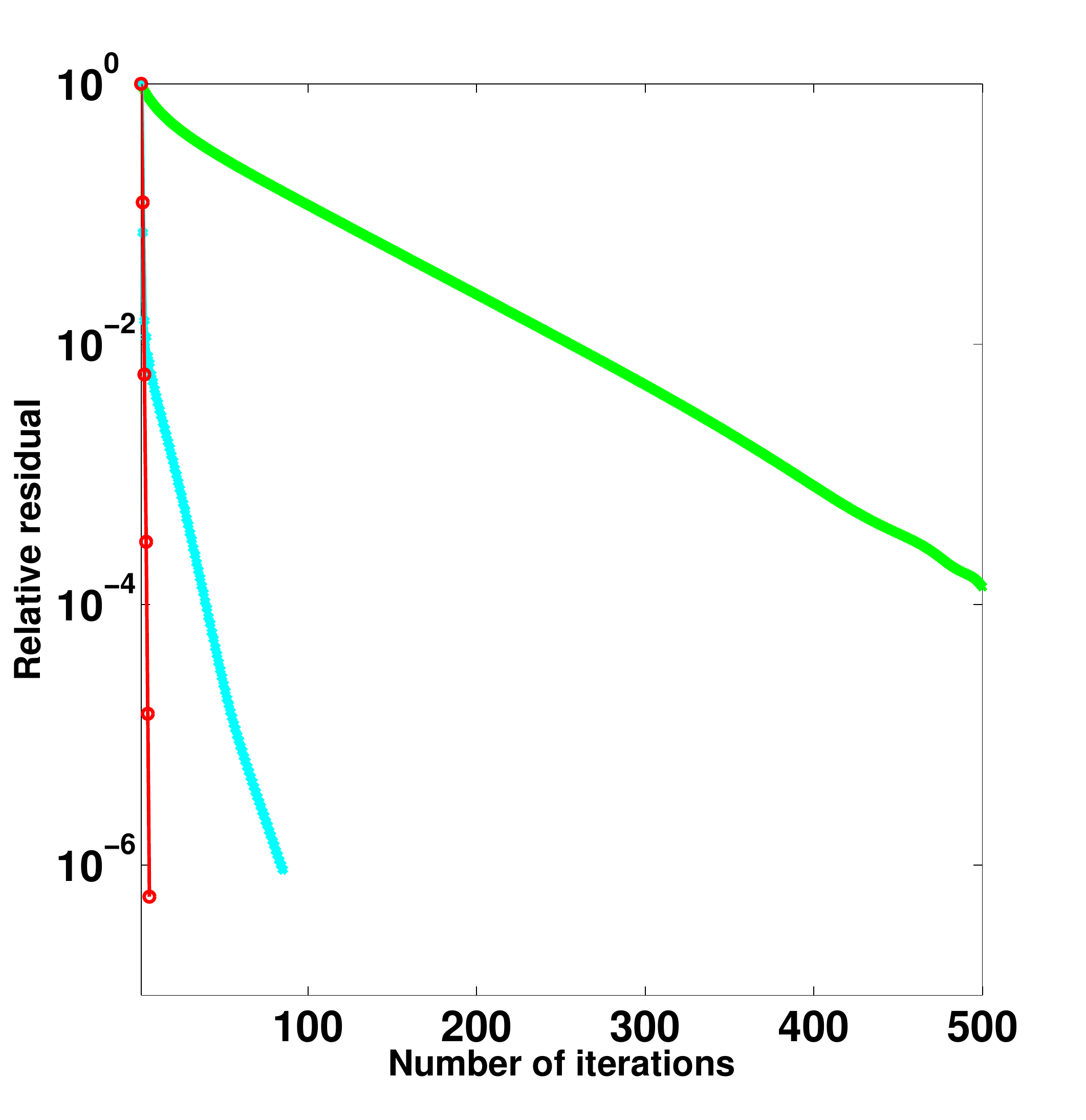}
\end{center}
\end{minipage} \\
\begin{minipage}[c]{1 \linewidth}
\hspace{1cm} Time grid 1
\hspace{2.5cm} Time grid 2
\hspace{2.5cm} Time grid 3
\end{minipage}
\vspace{-0.2cm}
\caption{Relative residual with GMRES for different time grids: GTP-Schur method with the local preconditioner (green), GTP-Schur method with the Neumann-Neumann preconditioner (cyan) and GTO-Schwarz method (red).} 	
	\label{A3Fig:CompresTimeRelres} 
\end{figure}

Now we analyze the error (in time) of the three algorithms for each of the three time grids. A reference solution is obtained by solving problem \eqref{A31dprobSub} - \eqref{A31dprobFrac} directly on a very fine time grid $ \Delta t = T/2000 $. The $ L^{2}-L^{2} $ error of the difference between the multi-domain and the reference solutions at each iteration is computed. We distinguish two different errors: error in the rock matrix $ L^{2}(0,T; L^{2}(\Omega_{i})), i=1,2, $ and error in the fracture $ L^{2}(0,T; L^{2}(\gamma)) $. Figures~\ref{A3Fig:CompresTimeErrorP} and \ref{A3Fig:CompresTimeErrorPfrac} show the pressure error in the subdomains and in the fracture respectively.

We first observe that the error in the subdomains after convergence (Figure~\ref{A3Fig:CompresTimeErrorP}) in the nonconforming case (Time grid~2) is equal to that in the conforming coarse case (Time grid~1) for all three algorithms. This is as expected as we use the same time step $ \Delta t_{m} = T/100 $ in the matrix for both of these grids. However, as already pointed out in Remark~\ref{A3rmrkOSWRaccuracy}, though one might hope that the error in the fracture (Figure~\ref{A3Fig:CompresTimeErrorPfrac}) in the nonconforming case is close to that in the conforming fine grid case (Time grid~3), this can only be the case for GTP-Schur method with the local preconditioner. Only for this case do we actually solve the fracture problem on the fine grid. For the other algorithms, the fracture error of the nonconforming case is equal to that of the conforming coarse grid instead (see Remark~\ref{A3rmrkNNprecond}). However, none of the methods deteriorates the accuracy because of nonconforming time grids. 
\begin{figure}[htbp]
\hspace{-0.3cm}\begin{minipage}[c]{0.3 \linewidth}
\centering
\includegraphics[scale=0.22]{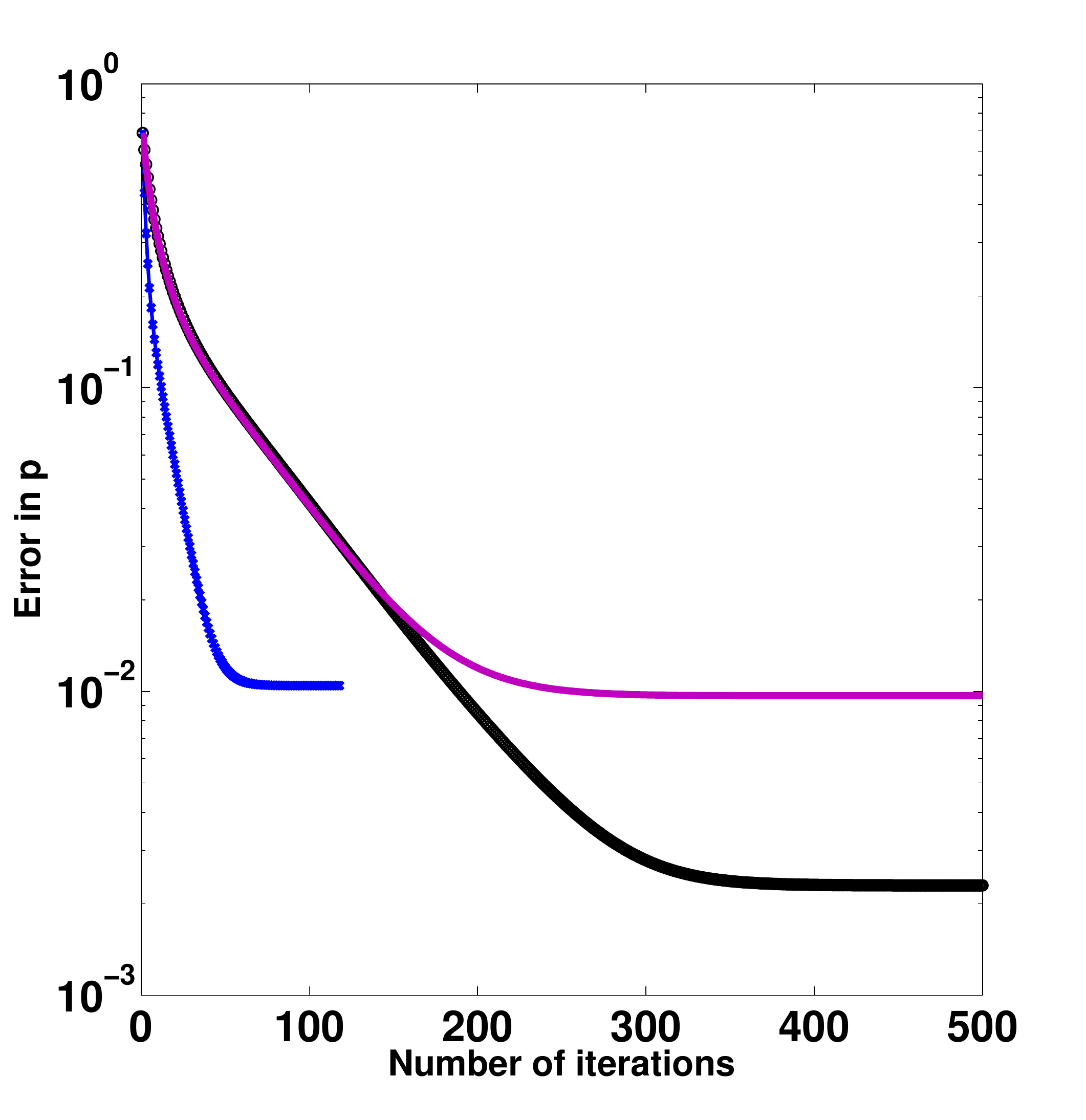}
\end{minipage} \hspace{10pt}
\begin{minipage}[c]{0.30 \linewidth}
\begin{center}
\includegraphics[scale=0.22]{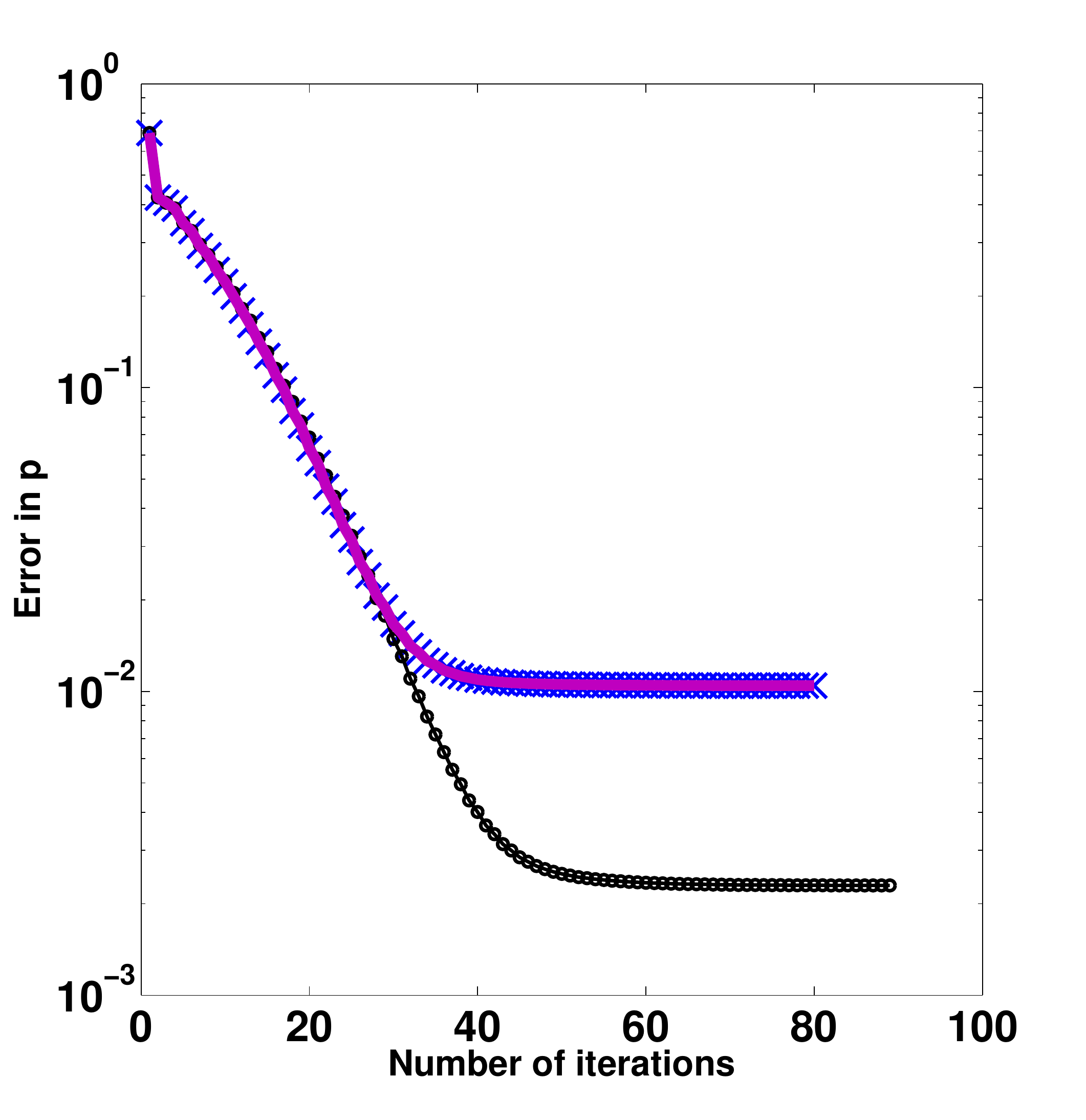}
\end{center}
\end{minipage} \hspace{10pt}
\begin{minipage}[c]{0.30 \linewidth}
\begin{center}
\includegraphics[scale=0.22]{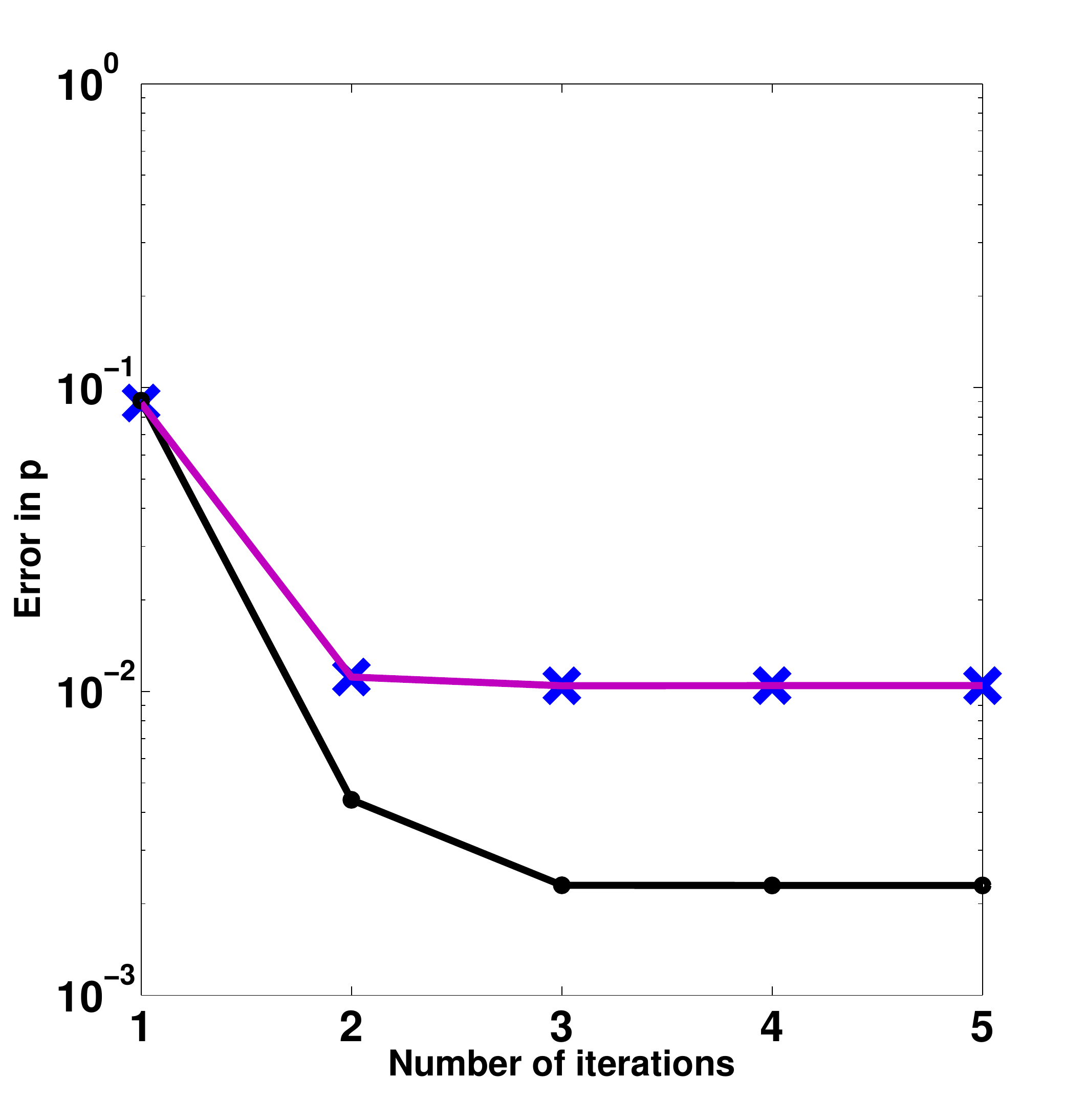}
\end{center}
\end{minipage} \\
\begin{minipage}[c]{1 \linewidth}
\hspace{0.3cm}  {\footnotesize GTP-Schur - local precond.}
\hspace{0.62cm} {\footnotesize GTP-Schur  - NN precond.}
\hspace{1.65cm} {\footnotesize GTO-Schwarz }
\end{minipage}
\vspace{-0.2cm}
\caption{$ L^{2} $ pressure error in the rock matrix: Time grid 1 (blue), Time grid 2 (magenta), Time grid 3 (black).} 	
	\label{A3Fig:CompresTimeErrorP} 
\end{figure}
%
%
%
\begin{figure}[htbp]
\hspace{-0.3cm}\begin{minipage}[c]{0.30 \linewidth}
\begin{center}
\includegraphics[scale=0.22]{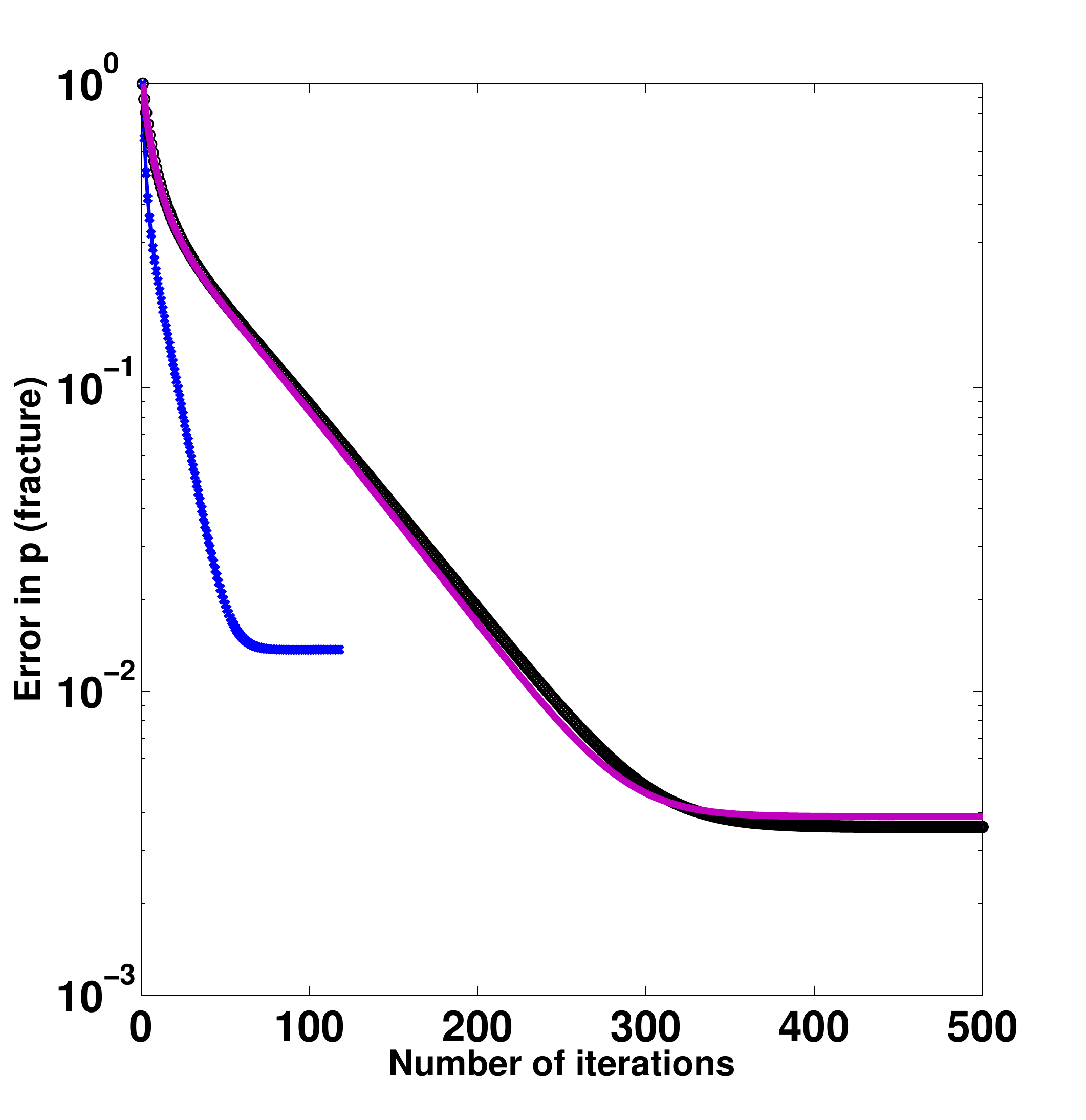}
\end{center}
\end{minipage} \hspace{10pt}
\begin{minipage}[c]{0.30 \linewidth}
\begin{center}
\includegraphics[scale=0.22]{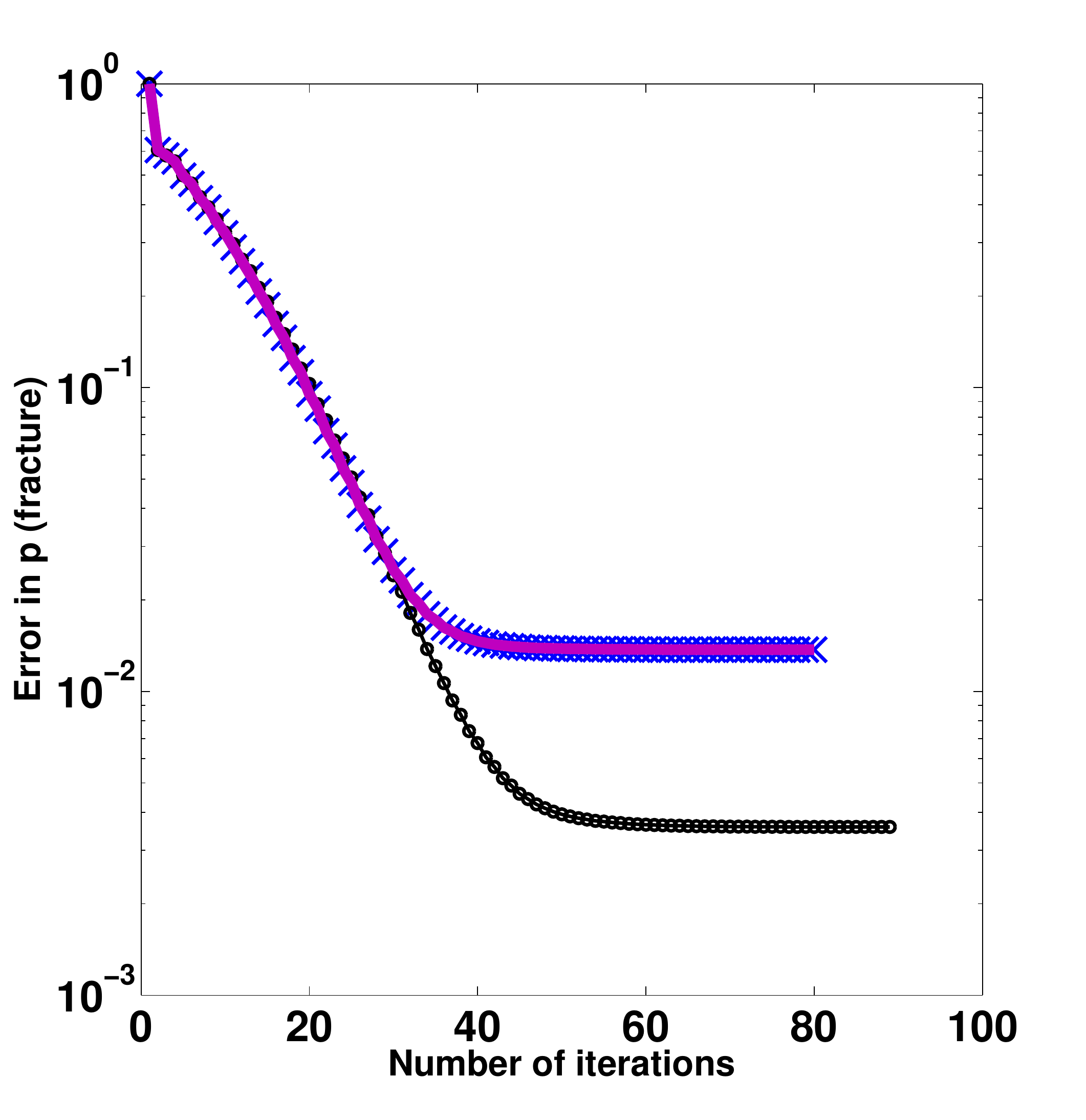}
\end{center}
\end{minipage} \hspace{10pt}
\begin{minipage}[c]{0.30 \linewidth}
\begin{center}
\includegraphics[scale=0.22]{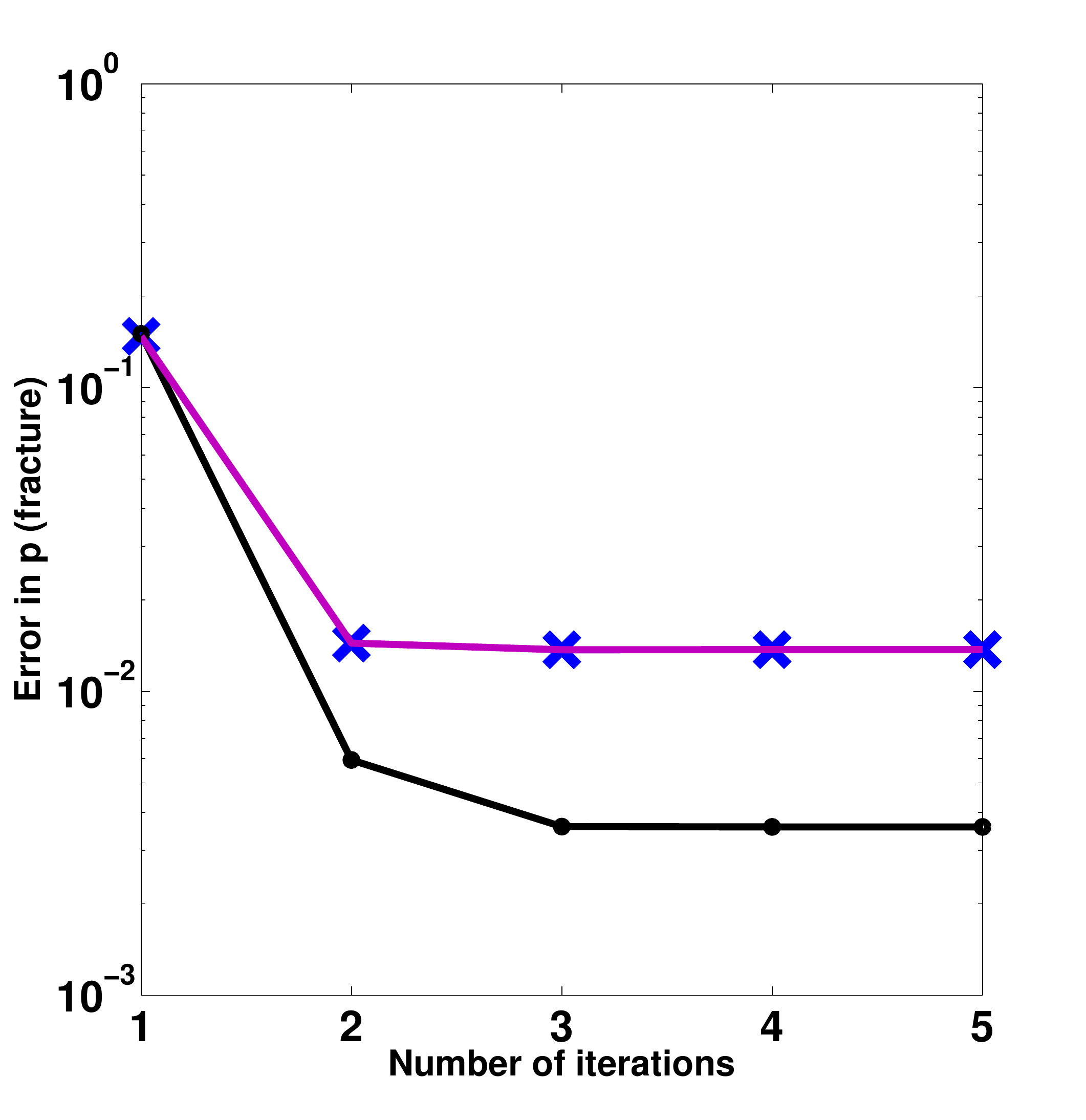}
\end{center}
\end{minipage} \\
\begin{minipage}[c]{1 \linewidth}
\hspace{0.3cm}  {\footnotesize GTP-Schur - local precond.}
\hspace{0.62cm} {\footnotesize GTP-Schur  - NN precond.}
\hspace{1.65cm} {\footnotesize GTO-Schwarz }
\end{minipage}
\vspace{-0.2cm}
\caption{$ L^{2} $ pressure error in the fracture: Time grid 1 (blue), Time grid 2 (magenta), Time grid 3 (black).} 	
	\label{A3Fig:CompresTimeErrorPfrac} 
\end{figure}
\begin{remark} \label{A3rmrkFracDD}
While the GTO-Schwarz method does not make it particularly useful to use a finer time grid in the fracture, it does give a rather remarkable convergence speed.  For the advection-diffusion problem with an explicit time scheme for advection, one of the main advantages of using smaller time steps in the fracture is to avoid imposing a time step in the two subdomains dictated by the CFL number of the equation in the fracture.  Thus we are hopeful that this algorithm will be useful when coupled with the advection equation simply for the convergence speed that it gives.  We add however that we are still pursuing some ideas for modifying this scheme to obtain an algorithm that can take advantage of smaller time steps in the fracture for the diffusion equation.  
\end{remark}

\section*{Conclusion}
We consider two domain decomposition methods for modeling the compressible flow in fractured porous media in which the fractures are assumed to be much more permeable than the surrounding medium. Two space-time interface problems are formulated using the time-dependent Dirichlet-to-Neumann and the Ventcell-to-Robin operators respectively, so that different time discretizations in the subdomains and in the fracture can be adapted. For the GTP-Schur method, two different preconditioners - the local and the Neumann-Neumann preconditioners- are considered and are first validated for a simple test case with one fracture. For the GTO-Schwarz method, the optimized parameter is used to accelerate the convergence of the associated iterative algorithm. Preliminary numerical experiments show that the GTO-Schwarz method converges much faster than the GTP-Schur method (with either the preconditioner) in terms of the number of iterations. The Neumann-Neumann preconditioner works better than the local preconditioner in the sense that its convergence is faster and is only weakly dependent on the mesh size of the discretizations. The GTO-Schwarz method also has a weak dependence on the mesh size. When nonconforming time steps are used, only the local preconditioner preserves the accuracy in time: the $ L^{2} $ error in the fracture of the nonconforming time grid is close to that of the conforming fine grid. For the other algorithms, the $ L^{2} $ error in the fracture of the nonconforming time grid is close to that of the conforming coarse grid instead.
However, for the GTO-Schwarz method, this weak point when different time steps are used is compensated by the fast
convergence of the algorithm.

\appendix 

\section{Proof of Theorem~\ref{thm:abst}}
\label{sec:append}

We now give the proof of Theorem~\ref{thm:abst}. 
The proof of the theorem is based on the Galerkin method, and its main
steps will be given after the following lemma, which states the main
energy estimates.
\begin{remark}
The proof of Lemma~\ref{lem:apriori}
is given in the infinite dimensional setting but some technical points (those involving
$\pmb{u}$ at time $t=0$) can only be defined using their finite
dimensional Galerkin approximations (as was done in detail for
Dirichlet and Robin boundary conditions in~\cite{PhuongThesis}). The
results presented below have to be understood in that sense.  
\end{remark}
\begin{lemma}
\label{lem:apriori}
  Under assumptions~\eqref{eq:H2}, \eqref{eq:H3} and \eqref{eq:H5} above, the following
  a priori estimates hold.
  \begin{equation}
    \label{eq:estim}
    \begin{aligned}
  \| p \|_{L^\infty(0, T; M)}^2 & \leq C \left( \| L \|^2_{L^2(0,T; M)²} +
      \|p_0\|_M^2 \right), \\ 
    \| \bu \|_{L^2(0, T; \Sigma_a)}^2 & \leq C \left( \| L \|^2_{L^2(0,T; M)²} +
      \|p_0\|_M^2 \right), \\
  \| \partial_t p \|_{L^2(0, T; M)}^2 & \leq C \left( \| L \|^2_{L^2(0,T; M)²} +
      \|p_0\|_W^2 \right), \\
   \|B\bu\|_{L^2(0, T; M)}^2 
   & \leq C
    \left( \| L \|^2_{L^2(0,T; M)²} +  \|p_0\|_W^2 \right),
    \end{aligned}
  \end{equation}
where we recall that $ \Sigma_a$ denotes the space $ \Sigma$ with the norm induced by the bilinear form $a$.
\end{lemma}

\begin{proof}
As usual we proceed by estimating successively $ p $,  $ \bu $ and $ \partial _{t} p $. \\

$\bullet$
First, to derive an estimate for $ p $, we take $ p(t) \in M $ and $ \bu(t) \in \Sigma $ as the test functions in \eqref{App31dWeak} and add the two equations to obtain
$$ a (\bu, \bu) + c(p, p)+   (\partial _{t} p, p)_M = L(p).
$$
Using the Cauchy-Schwarz inequality, we see that
\begin{equation} \label{App3lmm1}
\frac{1}{2} \frac{d}{dt} \| p \|_{M}^{2} + c(p, p) +  \| \bu \|_{\Sigma_a}^{2} \leq \frac{1}{2}\left (\| L\|_{M}^{2} + \| p\|_{M}^{2}\right ),
\end{equation}
Now integrating \eqref{App3lmm1} over $ (0,t) $ for $ t \in (0,T] $, we find
$$ \| p(t) \|_{M}^{2} + 2\int_0^t c(p, p) + 2 \int_{0}^{t} \| \bu \|_{\Sigma_a}^{2}  \leq \| p_{0}\|_{M}^{2}+\| L \|_{L^{2}(0,T; M)}^{2} + \int_{0}^{t} \| p\|_{M}^{2}.
$$
Then we use the non-negativity of $c$ and apply Gronwall's lemma to obtain
the first two estimates in \eqref{eq:estim}
$$ \| p \|_{L^{\infty}(0,T; M)}^{2} \leq C \left (\| p_{0}\|_{M}^{2}+\| L \|_{L^{2}(0,T; M)}^{2} \right ),
$$
and 
\begin{equation}\label{App3EstU}
\|\bu \|_{L^{2}(0,T; \Sigma_a)}^{2} \leq C\left (\| p_{0}\|_{M}^{2}+\| L \|_{L^{2}(0,T; M)}^{2} \right ).
\end{equation}
$\bullet$
Next, to estimate $ \partial _{t} p $, we differentiate the first equation of \eqref{App31dWeak} with respect to $ t $ and take $ \bu $ as a test function. This yields
\begin{equation} \label{App3lmm2a}
a(\partial_{t} \bu, \bu) - b(\bu, \partial_{t} p) = 0.
\end{equation}
Then taking $ \partial_{t} p $ as a test function in the second equation of~\eqref{App31dWeak} wee see that
\begin{equation} \label{App3lmm2b}
(\partial_{t} p , \partial_{t} p )_M + c(p, \partial_t p) + b(\bu, \partial_{t} p) = L(\partial_{t} p). 
\end{equation}
Now adding \eqref{App3lmm2a} and \eqref{App3lmm2b}, we obtain
$$ a(\partial_{t} \bu, \bu) + (\partial_{t} p , \partial_{t} p )_M +
c(p, \partial_t p) = L (\partial _{t} p),
$$
or 
\begin{equation}\label{App3lmm2c-1}
\| \partial _{t} p \|_{M}^{2} + \frac{1}{2} \frac{d}{dt} c(p, p) + \frac{1}{2} \frac{d}{dt} \| \bu \|_{\Sigma_a}^{2}  \leq \frac{1}{2}\| L\|^{2}_{M} + \frac{1}{2} \| \partial _{t} p \|_{M}^{2}.
\end{equation}
\medskip
Integrating this inequality over $ (0,t) $ for $ t \in (0,T] $, we have
\begin{equation}\label{App3lmm2c}
 \int_{0}^{t}\| \partial _{t} p \|_M^{2}+ c(p(t), p(t)) +  \| \bu (t)
 \|_{\Sigma_a}^{2} \leq \|L\|^{2}_{L^{2}(0,T; M)} + C_c \| p_0 \|_M^2 +  \| \bu (0) \|_{\Sigma_a}^{2},
\end{equation}
where $C_c$ is the constant of continuity of the bilinear form $c$.
There remains to bound the term $  \| \bu (0) \|_{\Sigma_a}^{2} $. Toward this end, we use the first equation of \eqref{App31dWeak} with $ \bv = \bu $ and for $ t=0 $:
\begin{equation*}
  a(\bu(0), \bu(0)) = b(\bu(0), p_0).
\end{equation*}
The (regularity) assumption that $p_0 \in W$ enables us to write 
\begin{equation*}
  \| \bu(0) \|_{\Sigma_a} \leq C_b \|p_0\|_W,
\end{equation*}
and, as $c(p(t), p(t)) \ge 0$,  this gives the third inequality in~\eqref{eq:estim}\\
%
%

$\bullet$
We now derive the last estimate. For this, we take $\mu = B \bu$ as
test function in the second equation of~\eqref{App31dWeak}.
\begin{equation*}
  (\partial_t p, B \bu)_M + c(p, B \bu) + b(\bu, B \bu) = (L, B \bu)_M
\end{equation*}
which we rewrite as 
\begin{equation*}
  \| B \bu \|_M^2 = (L -\partial_t p, B \bu)_M -c(p, B \bu)
  \leq C \left(
    \|L\|_M^2 + \| \partial_t p \|_M^2  +\|  p \|_M^2\right) +   \dfrac{1}{2} \| B \bu \|_M^2,
\end{equation*}
and the fourth inequality then follows by integrating in time and using the previous inequalities, which
completes the proof of the lemma.
\end{proof}

\medskip
We now give the proof of the theorem.
\begin{proof}
We first prove an estimate for $\| \bu \|_\Sigma$, which follows easily from the second and
fourth inequalities in Lemma~\ref{lem:apriori} and
hypothesis~\eqref{eq:H4p}:
\begin{equation}
\label{eq:estimu}
\beta \| \bu \|_{L^2(0, T; \Sigma)}^2 \leq C\int_0^T \left(  \| \bu
  \|_{\Sigma_a }^2 +
\| B \bu \|_M^2 \right) \leq C \left( \| L \|^2_{L^2(0,T; M)²} +
      \|p_0\|_W^2 \right).
\end{equation}
Note that this is the only place in the proof where~\eqref{eq:H4p} was
used. Lemma~\ref{lem:apriori} is independant of this
hypothesis.

With the a priori estimates from Lemma~\ref{lem:apriori} and~\eqref{eq:estimu},
the proof is concluded by using Galerkin's method.
\end{proof}

\pagebreak

 \bibliographystyle{plain}
\bibliography{Fracturespaper} %

\end{document}